\documentclass[11pt]{amsart}

\usepackage{amsmath}
\usepackage{amssymb}
\usepackage{amsbsy}
\usepackage{amscd}
\usepackage{amsfonts}
\usepackage{amsopn}
\usepackage{amstext}
\usepackage{amsthm}
\usepackage{amsxtra}
\everymath={\displaystyle}

\newcommand{\beq}{\begin{equation}}
\newcommand{\eeq}{\end{equation}}

\newcommand{\bee}{\begin{eqnarray}}
\newcommand{\ene}{\end{eqnarray}}
\newcommand{\bea}{\begin{eqnarray*}}
\newcommand{\ena}{\end{eqnarray*}}

\newcommand*{\R}{\ensuremath{\mathbb{R}}}

\newcommand*{\Z}{\ensuremath{\mathbb{Z}}}

\newcommand{\ee}{\varepsilon}

\newcommand{\rad}{\operatorname{rad}}

\newcommand{\op}{\mathrm{op}}
\newcommand{\1}{\mathbf{1}}

\newcommand{\f}{\frac}

\theoremstyle{plain}
\newtheorem{theo}{Theorem}[section]
\newtheorem{lemma}[theo]{Lemma}

\newtheorem{prop}[theo]{Proposition}

\newtheorem{rem}[theo]{Remark}

\theoremstyle{definition}

\numberwithin{equation}{section}

\makeatletter
\@namedef{subjclassname@2020}{\textup{2020} Mathematics Subject Classification}
\makeatother

\begin{document}

\title[]{New Bounds for $\mathrm{GL}(3)\times\mathrm{GL}(2)$ Shifted Convolution Sums}

\date{\today}

\author[G. W. Hu]{Guangwei Hu}
\address{School of Mathematics and Statistics, Shandong Normal University, Jinan, Shandong 250358, China} \email{gwhu@sdnu.edu.cn}
\author[G. S. L\"{u}]{Guangshi L\"{u}}
\address{School of Mathematics, Shandong University,
Jinan 250100, China}\email{gslv@sdu.edu.cn}

\maketitle

\begin{abstract}

Shifted convolution sums of $\mathrm{GL}(3)\times\mathrm{GL}(2)$ type have been extensively studied in several important settings, while a number of natural cases remain open. In this paper, we first revisit the shifted convolution sums involving $\lambda_\pi(1,n)$ and the divisor function $d(n)$, improving the recent power saving of $1/42$ obtained by Blomer and Li to $1/22$. We also treat the corresponding sums involving $\lambda_\pi(1,n)$ and $r(n,\mathcal{Q})$, settling a previously open case and obtaining a power saving of $1/22$. Here $\lambda_\pi(1,n)$ denotes the normalized Fourier coefficients of a Hecke--Maass cusp form $\pi$ on $\mathrm{SL}(3,\mathbb{Z})$, and $r(n,\mathcal{Q})$ denotes the number of representations of a positive integer $n$ by a positive definite integral binary quadratic form $\mathcal{Q}$.

The proof combines two circle methods with a different organization of the major-arc analysis. A key point is that the nonzero Poisson frequencies are kept together in a single coefficient vector, rather than estimated separately, which allows them to be controlled by a finite matrix norm. The zero Poisson frequency is treated separately: its contribution is independent of the reduced residue class, and its arithmetic mean cancels the constant term in the complementary Jutila weight before absolute values are taken. We make this cancellation explicit, including the relevant normalization and local factors. The same mechanism extends to the binary quadratic-form setting, while the remaining Jutila contribution is handled by the available correlation estimates.
\end{abstract}

\section{Introduction}

Shifted convolution sums form a foundational topic in number theory, as they quantify correlations between two arithmetic functions evaluated at arguments differing by a fixed additive shift. More precisely, given two arithmetic sequences $\alpha$ and $\beta$, a fundamental summation object frequently encountered in number theory takes the form
\begin{align*}
\sum_{n\leq X} \alpha(n)\beta(n + h).
\end{align*}
Here $X$ denotes the upper bound of the summation variable $n$, and $h$ is a fixed nonzero integer referred to as the shift.  The primary goal of studying this shifted convolution sum is to quantify the correlation between two arithmetic functions, which can be accomplished by establishing a nontrivial power saving in $X$. Shifted convolution sums are deeply intertwined with central problems in the subject, including moments of automorphic $L$-functions, subconvexity estimates, equidistribution questions, and the quantum unique ergodicity conjecture.

Shifted convolution sums manifest themselves in many different forms, with arithmetic weight functions arising from a broad collection of classical objects. Prominent examples include sums involving the von Mangoldt function, the M\"obius function, divisor functions, and Fourier coefficients of $\mathrm{GL}(m)$ cusp forms, together with numerous other multiplicative arithmetic sequences extensively studied in the literature.

One classical and extensively investigated family of shifted convolution sums is built from Fourier coefficients of $\mathrm{GL}(2)$ cusp forms. Over recent decades, a rich and comprehensive theoretical framework has been developed for $\mathrm{GL}(2)\times\mathrm{GL}(2)$ shifted convolution sums, yielding definitive results for both holomorphic and Maass cusp forms. For more details concerning this topic as well as the body of relevant developments, we refer the reader to
\cite{Blomer-2004-IMRN,Blomer--Harcos-2008-Duke-Math,Duke-Friedlander-Iwaniec--1993-Invent,Duke-Friedlander-Iwaniec--1994-Invent,Harcos-2003-Math-Ann,Harcos-Michel,Holowinsky-2009-Duke-Math,Jutila-1996-Math-Z,KMV,Michel-2004-Ann-Math,Sarnak-2001-J-Funct-Anal,Selberg-1965} and the references therein.

Serious obstacles remain when one investigates shifted convolution sums attached to higher rank groups. Even the shifted convolution sums associated with $\mathrm{GL}(3)\times\mathrm{GL}(2)$ exhibit this sharp rise in complexity. The intricate structure of higher rank automorphic forms, combined with the shortage of effective tools for handling multidimensional character sums appearing in the course of analysis, renders such problems considerably more challenging.

On the other hand, there are many variants of $\mathrm{GL}(3)\times\mathrm{GL}(2)$ shifted convolution sums. Different types of these sums tend to call for quite different methods, which further adds to the difficulty of the problem. In the setting where one component corresponds to the triple divisor function $d_3(n)$ (the Fourier coefficient attached to a $\mathrm{GL}(3)$ Eisenstein series), the factorability of the underlying arithmetic function plays a central role in the analysis. By virtue of the favorable structural properties of the triple divisor function, several results on its associated shifted convolution sums have been established in recent years. For further details, see \cite{Hooley,Munshi-2013-divisor,Pitt-1995-Duke,Tang-2020-BAMS,Topacogullari-2016}.

Parallel to the non‑cuspidal setting of $\mathrm{GL}(3)$, the cuspidal case with arithmetic weight $\lambda_\pi(1,n)$, the normalized Fourier coefficients of a Hecke--Maass form $\pi$ on $\mathrm{SL}(3,\mathbb{Z})$, has likewise seen substantial progress. In this setting, however, one cannot make use of the factorability of the underlying arithmetic function.

A canonical example is provided by shifted convolution sums involving Fourier coefficients of automorphic forms on $\mathrm{GL}(3)$ and $\mathrm{GL}(2)$, that is,
\begin{align}\label{SCS-Munshi}
\sum_{n} \lambda_\pi(1,n)\lambda_f(n + h)\phi(n),
\end{align}
where $\lambda_f(n)$ denotes the normalized Fourier coefficients of a Hecke--Maass or  Hecke holomorphic cusp form $f$ on $\mathrm{SL}(2,\mathbb{Z})$, and $\phi$ is a smooth function compactly supported on $[X,2X]$. By the Cauchy--Schwarz inequality and Rankin--Selberg theory, one obtains the trivial bound $O(X^{1+\varepsilon})$. By exploiting the flexibility of factorable moduli in Jutila's circle method, Munshi \cite{Munshi-Duke2013} obtained a power saving of $1/26$ in $X$, after refining the relevant exponential-sum estimates as noted by Xi \cite{xi-shifted-sum-Forum}. Later, Xi \cite{xi-shifted-sum-Forum} further developed this approach by incorporating the theory of arithmetic exponent pairs for algebraic trace functions. He obtained the improved bound $O(X^{21/22+\varepsilon})$, corresponding to a power saving of $1/22$.

Another example of interest is a $\mathrm{GL}(3)\times\mathrm{GL}(2)$ shifted convolution sum where the Fourier coefficient $\lambda_f(n)$ appearing in \eqref{SCS-Munshi} is replaced by the divisor function $d(n)$.   In this setting, a direct application of Jutila's circle method encounters an additional difficulty, since the exponential sum involving $d(n)$ does not by itself yield the required saving. Recently, Blomer and Li \cite{Blomer-Li-2025} resolved this problem by combining two distinct delta-symbol methods. Their treatment uses Jutila's circle method with factorable moduli to create the bilinear structure required for one part of the argument, together with a Kloosterman refinement of the circle method for the complementary contribution.
In this way, they obtained a power saving of $1/42$.

This is the first problem we revisit in the present paper. A different organization of the major-arc analysis allows us to improve the power saving from $1/42$ to $1/22$.  Our main result is as follows.

\begin{theo}\label{B-L}
Let $X\geq 1$ and  $h$ be a positive integer with $h\leq X$. Let $W$ and $W_0$ be smooth functions compactly supported on $[X,2X]$ satisfying $x^jW^{(j)}(x)\ll_j1$, $x^jW_0^{(j)}(x)\ll_j1$, and let $\lambda_\pi(1,n)$ denote the normalized Fourier coefficients of a Hecke--Maass cusp form $\pi$ for $\mathrm{SL}(3,\mathbb{Z})$. Then, for any $\ee>0$, we have
\begin{align*}
\sum_{m-n=h}\lambda_\pi(1,n)d(m)W_0(m)W(n)\ll X^{1-\f{1}{22}+\ee},
\end{align*}
uniformly in $h$, where  the implied constant depends on $\pi$, $W$, $W_0$ and $\ee$.
\end{theo}

For simplicity, we formulate and prove the theorem for the relation $m-n=h$. The argument can be extended to the more general relation
$\lambda_1m-\lambda_2n=h$, where $\lambda_1,\lambda_2\in\mathbb{Z}\setminus\{0\}$ fixed, with the same power saving of $1/22$. This requires us to handle numerous $\gcd$ constraints, as in \cite{Blomer-Li-2025}.
In that setting, the shortening of the corresponding summation ranges may also yield additional savings in $\lambda_1$ and $\lambda_2$. Since keeping track of this dependence would introduce further technical details without affecting the main ideas of the argument, we restrict throughout to the case $\lambda_1=\lambda_2=1$.

We also consider another arithmetic function, namely $r(n,\mathcal{Q})$, the number of representations of a positive integer $n$ by a positive definite integral quadratic form $\mathcal{Q}$ in $l$ variables. More precisely, we consider sums of the form
\begin{align*}
	\sum_n \lambda_\pi(1,n+h)r(n,\mathcal{Q})\phi(n),
\end{align*}
where $\phi$ is a smooth function compactly supported on $[X,2X]$. By the Cauchy--Schwarz inequality, the individual bound for $r(n,\mathcal{Q})$, and Rankin--Selberg theory, one obtains the trivial bound
$O(X^{l/2+\varepsilon})$. In 2019, Jiang and L\"u \cite[p.375]{JL} used an asymptotic formula for $r(n,\mathcal{Q})$ to obtain nontrivial upper bounds for $l\geq 4$. In forthcoming work, we treat the case $l=3$ and obtain a power saving of $1/6$; further details will appear in \cite{HL-preprint}. Quadratic forms in fewer variables are of particular arithmetic interest, while the corresponding shifted convolution problems become significantly more difficult. The binary quadratic-form case, corresponding to $l=2$, is particularly challenging and has remained open.

In view of the structural similarities between the representation function $r(n,\mathcal{Q})$ and the divisor function $d(n)$, it is natural to ask whether analogous methods can yield nontrivial upper bounds of comparable strength for $\mathrm{GL}(3)\times\mathrm{GL}(2)$ shifted convolution sums involving $r(n,\mathcal{Q})$. This structural parallel suggests that the method developed for the divisor-function setting can be adapted, with suitable modifications, to the quadratic-form setting. Motivated by this observation, we resolve the binary quadratic-form case and obtain a power saving of $1/22$. Our second main result is as follows.

\begin{theo}\label{T1}
	Let $\mathcal{Q}(\mathbf{x})$ be a positive definite integral quadratic form in two variables, and let $r(n,\mathcal{Q})$ denote the number of representations of $n$ by $\mathcal{Q}(\mathbf{x})$. Let $\lambda_\pi(1,n)$ denote the normalized Fourier coefficients of a Hecke--Maass cusp form $\pi$ for $\mathrm{SL}(3,\mathbb{Z})$. Suppose that $\phi$ is a smooth function compactly supported on $[X,2X]$  satisfying $x^j\phi^{(j)}(x)\ll_j1$,  and let $h$ be a positive integer with $h\leq X$. Then, for any $\varepsilon>0$, we have
	\begin{align*}
		\sum_n
		\lambda_\pi(1,n+h)r(n,\mathcal{Q})\phi(n)
		\ll		X^{1-\f{1}{22}+\varepsilon},
	\end{align*}
	where  the implied constant depends on $\pi$, $\mathcal{Q}$, $\phi$ and $\ee$.
\end{theo}

As a direction for future research, the present approach may lead to improved upper bounds in asymptotic formulas associated with twisted moments of $\mathrm{GL}(3)$ $L$-functions. The flexibility of the method also suggests possible extensions to a broader class of related problems. We leave these questions and further applications of the present strategy for future work.

It may be useful to indicate some of the main features behind the proof. Inspired by Blomer and Li \cite{Blomer-Li-2025}, we decompose the shifted convolution sum into two contributions. The first is treated essentially by the available correlation estimates, whereas the main improvement comes from a different analysis of the complementary contribution. 

Unlike Blomer and Li \cite[Section 3.1]{Blomer-Li-2025}, who handle the constant term, zero circle frequency, and nonzero circle frequency separately, we arrange the parameters so that the zero circle frequency vanishes identically. Upon Poisson summation, the constant term is grouped together with the zero Poisson frequency before taking absolute values. The nonzero Poisson frequencies are instead retained as a single packed family and treated through the matrix norm; see Section \ref{sec-S2-BL}. The main saving of our work arises from this treatment.

A crucial point is the treatment of the nonzero Poisson frequencies. The relevant complete character sums admit a finite matrix representation together with a uniform operator norm bound. The gain needed here comes from applying this matrix formulation before the auxiliary variables are separated.  After Poisson summation, the relevant character sums are naturally indexed by rows of the form
\begin{align*}
	d_0-g_1l\nu\bar p \bmod c.
\end{align*}
Writing $k=l\nu$, we keep the whole $(p,k)$-family in a single coefficient vector rather than estimating the frequencies separately. The short support of the parameters makes the corresponding row map injective. This allows us to apply the finite matrix norm directly to the packed family before separating the auxiliary variables. In particular, the nonzero-frequency contribution is reduced to	$O({X^{1+\ee}C_0}/{Q})$, which is one of the main gains in the major-arc analysis.

On the other hand, the zero Poisson frequency exhibits a different structure. After recombining the divisor variables and the dyadic pieces, its multiplier $\chi_0(c,z)$ becomes independent of the reduced residue class. The elementary mean
\begin{align*}
	a_p(n)
	=
	\sum_{\substack{r\mid n\\(r,p)=1}}
	\frac{\mu(r)}{r},
	\qquad
	\sum_{n\leq T}a_p(n)
	=
	\frac{T}{\zeta(2)(1-p^{-2})}
	+O(\log(2T)),
\end{align*}
then yields
\begin{align*}
	1-\chi_0(c,z)	\ll
	X^\ee	\Big(\frac{1}{cQ\delta}	+\frac{1}{Q_1}+\frac{1}{Q_2}	\Big).
\end{align*}
The important point is that the constant term and the zero-frequency contribution are combined before absolute values are taken. This arithmetic cancellation removes the term that would otherwise dominate the major-arc estimate.

The same mechanism extends to the binary quadratic-form problem. There, the normalized quadratic Gauss sum acts as a bounded multiplier, so the matrix argument for the nonzero Poisson frequencies remains available, while the zero-frequency analysis is modified only by finitely many local factors. Consequently, the two problems lead to the same effective parameter balance. 
 With an appropriate choice of parameters, these estimates yield the power saving $1/22$ in both cases.

\section{Preliminaries}

In this section, we briefly recall several fundamental definitions and standard results concerning the two delta-symbol methods, the Voronoi summation formulae, and the Fourier coefficients of $\mathrm{SL}(3,\mathbb Z)$ Maass forms. 
We also record the character-sum estimates of Blomer--Li needed for $S_1$.  At the end of the section, we isolate the matrix structure of the complete character sum.

As usual, we write $e(z)=e^{2\pi iz}$ and $e_c(z)=e(z/c)$. A star on a residue-class sum means that the summation is restricted to reduced residue classes. The letter $\ee>0$ may be arbitrarily small and its value may vary from line to line.  
For the truncations below, we choose auxiliary positive constants $\ee_0,\ee_1,\eta_0$ sufficiently small in terms of the final $\ee$, with $\ee_1<\ee_0/10$, and put $\delta=X^{-1+\ee_0}.$
The resulting harmless powers will be absorbed into $X^\varepsilon$.

\subsection{Delta symbol methods}
In this subsection, we introduce two delta symbol methods. The first, due to Jutila \cite{Jutila-1992,Jutila-1996-Math-Z}, is highly flexible, though it only provides an approximation to the constant function in the $L^2$-sense. 

\begin{lemma}\label{Jutila circle method}
Let $Q \ge 1$,  and $w$ be a non-negative weight supported on $1\le q\le Q$. 
Let $\psi\colon [-1,1]\to [0,1]$ be a smooth function with $\smallint \psi=1$, and assume $0<\delta<1/2$.
Suppose that $L=\textstyle \sum_{q} \varphi(q)w(q) \neq 0$.
For $\alpha\in\mathbb{R}$, we define the $1$-periodic function
\begin{align*}
\chi(\alpha) = \frac{1}{\delta L} \sum_{q} w(q)
\sideset{}{^{*}}\sum_{a\bmod q}
\sum_{k\in\mathbb Z} \psi\Big(\frac1\delta\Big(\alpha - \frac{a}{q} + k\Big)\Big).
\end{align*}
Then we have
\begin{align*}
\int_0^1 \big(1-\chi(\alpha)\big)^2 d\alpha \ll \frac{Q^2 \|w\|^2_\infty |\log\delta|^3}{L^2 \delta}.
\end{align*}
\end{lemma}

Indeed, the $l$-th Fourier coefficient of $\chi$ equals
\begin{equation}\label{eq:fourier-chi}
\frac1L\sum_{q} w(q) r_q(l)\hat\psi(\delta l)
\begin{cases}
= 1, & l=0,\\
\ll_\psi \dfrac{Q}{L}(1+\delta|l|)^{-10}d(|l|)\|w\|_\infty, & l\neq 0,
\end{cases}
\end{equation}
where $\hat\psi$ denotes the Fourier transform and $r_q(l)$ the Ramanujan sum.

We further derive the following useful representation for $\alpha =  b/c + z$ where $b,c\in\mathbb Z$, $c\neq 0$ and $z\in\mathbb R$.
Opening the Ramanujan sum and applying the Poisson summation formula, we have
\begin{equation}\label{chi-repr}
\begin{aligned}
\chi\Big(\frac bc + z\Big)
&= \frac1L\sum_{q} w(q)\sum_{d\mid q} d\mu\Big(\frac qd\Big)
\sum_l \hat\psi(\delta ld) e\Big(\Big(\frac bc + z\Big)l d\Big)\\
&= \frac{1}{\delta L}\sum_{q} w(q)\sum_{d\mid q}\mu\Big(\frac qd\Big)
\sum_{l\equiv bd\bmod c} \psi\Big(\frac1\delta\Big(\frac{l}{cd}+z\Big)\Big).
\end{aligned}
\end{equation}

The second delta symbol corresponds to a Kloosterman-type refinement of the circle method constructed in the style of Heath-Brown \cite[Section 3]{hb}, and further detailed arguments can be found in \cite[Lemma 2 and (2.5)]{Blomer-Li-2025}.

\begin{lemma}\label{Kloosterman circle method}
Let $C\ge1$ and $n\in \mathbb{Z}$.  Then, we have
\begin{equation*}
\delta_{n=0}=\sum_{c\le C}\;\sideset{}{^{*}}\sum_{b\bmod c}\int_{-\f{1}{c(c+c')}}^{\f{1}{c(c+c'')}}e\Big(\Big(\f{b}{c}+z\Big)n\Big)\,dz,
\end{equation*}
where  $b'/c'<b/c<b''/c''$ are consecutive terms in the Farey sequence of order $C$. Furthermore, put
\begin{equation*}	
I(c,z)=\begin{cases}
	\Big[-\min\Big(C,\frac1{c|z|}-c\Big),\,-(C-c)\Big), & z<0,\\
	\Big(C-c,\,\min\Big(C,\frac1{c|z|}-c\Big)\Big], & z> 0,
\end{cases}
\end{equation*}
for $z\ne0$. Then each Farey arc may be rewritten in the form
\begin{equation}\label{eq:refined-farey}
\int_{-1/(cC)}^{1/(cC)}\frac1c\sum_{u\bmod c}\;\sum_{t\in I(c,z)}e\Big(\f{u(t+\bar b)}{c}\Big)e\Big(\Big(\f{b}{c}+z\Big)n\Big)\,dz,
\end{equation}
where the value at $z=0$ is immaterial. In particular, for $|u|\le c/2$, we have
\begin{equation}\label{eq:farey-t-sum}
\sum_{t\in I(c,z)}e_c(ut)\ll \frac{c}{1+|u|}.
\end{equation}
\end{lemma}

\subsection{Preliminaries on ${\rm SL}(3,\mathbb{Z})$ Maass forms}
We shall recall some basic facts about ${\rm SL}(3,\mathbb{Z})$ automorphic forms.  Suppose that $\pi$ is a Maass form of type $(\nu_1,\nu_2)$ for ${\rm SL}(3,\mathbb{Z})$ which is an eigenfunction of all the Hecke operators with Fourier coefficients $\lambda_\pi(n_1,n_2)$, normalized so that $\lambda_\pi(1,1)=1$ (for details see Goldfeld’s book \cite{goldfeld}). We introduce the Langlands parameters $(\alpha_1,\alpha_2,\alpha_3)$ defined by 
$$\alpha_1=-\nu_1-2\nu_2+1,\quad  \alpha_2=-\nu_1+\nu_2\quad \text{ and } \quad\alpha_3=2\nu_1+\nu_2-1.$$
The Ramanujan--Selberg conjecture predicts that $|\Re(\alpha_i)|=0$, and from the work of
Jacquet and Shalika \cite{js}, we  know that $|\Re(\alpha_i)|\leq 1/2.$ Let $g$ be a smooth function compactly supported  on $(0,\infty)$.  For $\sigma>-1+\max\{-\Re(\alpha_1),-\Re(\alpha_2), -\Re(\alpha_3)\}$, we define
\begin{equation*}
G^{\pm}(x)=\frac1{4\pi^{{5}/{2}} i}\int_{\Re s=\sigma}\tilde{g}(-s)(\pi^3x)^{-s}
\left(\prod_{j=1}^3\frac{\Gamma\Big(\frac{1+s+\alpha_j}{2}\Big)}{\Gamma\Big(\frac{-s-{\alpha}_j}{2}\Big)}\mp i \prod_{j=1}^3\frac{\Gamma\Big(\frac{2+s+\alpha_j}{2}\Big)}{\Gamma\Big(\frac{1-s-\alpha_j}{2}\Big)}\right)ds.
\end{equation*}
Let $$S(a,b,c)=\sideset{}{^{*}}\sum_{t\bmod c}e\Big(\f{at+b\bar{t}}{c}\Big)$$ be the Kloosterman sum, where $\bar{t}$ denotes the multiplicative inverse of $t\bmod{c}$.

The  following Voronoi summation formula (see \cite{GL} or \cite{MS}) will play an important role in our analysis.

\begin{lemma}\label{voronoi3}
Let $g$ be a smooth function compactly supported on $(0,\infty)$. For $(a,q)=1$, we have
\begin{equation}\label{vo3}
\begin{aligned}
\sum_{n=1}^\infty& \lambda_\pi(1,n)e_q({{a}n})g(n)=q\sum_{\pm}\sum_{n_1|q}\sum_{n_2=1}^\infty \frac{\lambda_\pi(n_2,n_1)}{n_1n_2}S\Big(\bar{a}, \pm n_2, \f{q}{n_1}\Big)G^{\pm}\Big(\frac{n_1^2n_2}{q^3}\Big),
\end{aligned}
\end{equation}
where $\bar{a}$ denotes the multiplicative inverse of $a\bmod{q}$, and $S(\bar{a}, \pm n_2, q/n_1)$ is the Kloosterman sum. 
\end{lemma}

\begin{rem}\label{rem-Gz}
If $g$ is supported on $[X, 2X]$, satisfying $x^j g^{(j)}(x) \ll_j 1$, then the sums on the right-hand side of \eqref{vo3} are essentially supported on $|n_1^2 n_2| \ll q^3 (qX)^\varepsilon / X$. The contribution from the terms with	$|n_1^2 n_2| \gg q^3 (qX)^\varepsilon / X$ is negligibly small. This follows by estimating the integral $G^{\pm}(y)$ by shifting the contour to the right. For smaller values of $n_1^2 n_2$, we shift the contour to the left up to $\sigma = -1/2+\ee$ (using the result of Jacquet and Shalika) to obtain
$$G^{\pm}\Big( \frac{n_1^2 n_2}{q^3} \Big) \ll X^{\ee}\Big(\frac{X n_1^2 n_2}{q^3}\Big)^{1/2}. $$
\end{rem}

Since the Ramanujan--Selberg conjecture for $\pi$ is still open, we resort to alternative estimates. The following lemma is well known.
\begin{lemma}\label{R-S}
We have 
$$\sum_{n\leq X}|\lambda_\pi(1,n)|^2\ll_{\pi,\ee} X^{1+\ee}.$$
\end{lemma}
Throughout the proof of our main theorem, we shall make extensive use of several summatory bounds for the Fourier coefficients $\lambda_\pi(n_1,n_2)$ associated with a cuspidal $\mathrm{GL}(3)$ automorphic representation $\pi$. All such bounds are deduced from the fundamental multiplicative identity satisfied by the Hecke eigenvalues of $\pi$,
\begin{align*}
\lambda_\pi(n_1,n_2) = \sum_{d\mid (n_1,n_2)} \mu(d) \lambda_\pi\Big(\frac{n_1}{d},1\Big) \lambda_\pi\Big(1,\frac{n_2}{d}\Big),
\end{align*}
in conjunction with classical Rankin--Selberg convolution theory (Lemma \ref{R-S}) and standard bounds that are available unconditionally without invoking the full Ramanujan conjecture for $\mathrm{GL}(3)$. More precisely, for all real parameters $X,Y\geq 1$ and any arbitrarily small positive real number $\varepsilon>0$, the following summation estimates hold uniformly
\begin{align*}
 \sum_{n\leq X}\sum_{m\leq Y} |\lambda_\pi(n,m)| \ll (XY)^{1+\varepsilon}, \qquad \sum_{m\leq Y} |\lambda_\pi(n,m)| \ll n^{1/2+\ee}Y^{1+\varepsilon}.
\end{align*}
Moreover, we deduce the following estimates
\begin{align}\label{lem:hecke-moments}
 \sum_{n\leq X}\sum_{m\leq Y} \f{|\lambda_\pi(n,m)|^2}{nm} \ll (XY)^{\varepsilon}, \quad \sum_{n\leq X} \f{|\lambda_\pi(n,m)|}{\sqrt n} \ll (Xm)^{1/2+\varepsilon}.
\end{align}

\subsection{The Voronoi summation formula for the divisor function}

The Voronoi summation formula is a foundational instrument in the analysis of shifted convolution sums, and it plays an important role throughout our work. For completeness, we state here the Voronoi summation formula for the divisor function.

\begin{lemma}\label{Voronoi-d(n)}
Let  $(b,c)=1$, and  $\omega$ be a smooth function  compactly supported on $(0,\infty)$. We have
\begin{align*}
\sum_{n}d(n)e\Big(\frac{bn}{c}\Big)\omega(n)
=&\frac1c\int_0^\infty \omega(\xi)\Big(\log\frac{\xi}{c^2}+2\gamma\Big)d\xi\\
&+\frac1c\sum_{\pm}\sum_{n}d(n)e\Big(\pm\frac{\overline{b}n}{c}\Big)\int_0^\infty \omega(\xi)J^\pm\Big(\frac{\sqrt{n\xi}}{c}\Big)d\xi,
\end{align*}
where $$J^{-}(\xi)=\sum_{\pm}e(\pm 2\xi)v_\pm(\xi), J^{+}(\xi)=v_0(\xi),$$ 
with
\begin{align*}
\xi^j v_\pm^{(j)}(\xi)\ll_j \frac{1+|\log\xi|}{1+\xi^{1/2}},\;v_0(\xi)\ll_A \frac{1+|\log\xi|}{1+\xi^{A}}
\end{align*}
for any $j,A\ge 0$.
\end{lemma}

For future use throughout our arguments, we first analyze the integral transform under the following special setting. Let $\omega$ be a smooth function compactly supported on $[X, 2X]$, satisfying $x^j \omega^{(j)}(x) \ll_j 1$.
An application of the method of stationary phase yields the bound
\begin{align}\label{eq:divisor-minus-transform}
\int_{0}^{\infty} \omega(\xi)e(z\xi)J^{\pm}\Big(\frac{\sqrt{n\xi}}{c}\Big)d\xi&\ll_{A,\varepsilon}
\frac{X(Xnc)^{\varepsilon}}{1+X|z|}\Big(1+\frac{nX}{c^{2}(1+X|z|)^{2}}\Big)^{-A},
\end{align}
with detailed discussions available in \cite[Lemma 4]{Blomer-Li-2025}, for example.

\subsection{Analysis of character sums}

The delta symbol method and the Voronoi summation formula give rise to a variety of character sums. It is essential to establish square-root cancellation for these sums, uniform with respect to several auxiliary parameters, at least in generic cases.
We first record the two estimates from Blomer--Li \cite{Blomer-Li-2025}. 

For each divisor $n_1\mid c$, let $q=c/n_1$. We define
\begin{align}	\label{eq:basic-character-sum}
S_{h,d,n_1,n_2}(c)=\sideset{}{^{*}}\sum_{b\bmod c}e_c(hb+\bar bd)S(\bar b,n_2;q).
\end{align}
By virtue of Weil’s bound for Kloosterman sums and the bounds of Adolphson and Sperber \cite{AS} based on Newton polygons, we obtain the following lemma. 

\begin{lemma}\label{lem:blomer-li-pointwise}
Let $c=c_{\mathrm{sf}}c_{\mathrm{sq}}$ with $c_{\mathrm{sf}}$  squarefree, $c_{\mathrm{sq}}$  squarefull and $(c_{\mathrm{sf}},c_{\mathrm{sq}})=1$. We have
\begin{align*}	
S_{h,d,n_1,n_2}(c)\ll c^{1+\ee}c_{\mathrm{sq}}^{1/2}\Big(\frac{(n_1,c_{\mathrm{sf}},d,h)}{n_1}\Big)^{1/2}\ll c^{1+\ee}c_{\mathrm{sq}}^{1/2}.
\end{align*}
\end{lemma}

This is Lemma 7 of Blomer and Li \cite{Blomer-Li-2025}. The improvement in the present paper does not strengthen the pointwise estimate in Lemma \ref{lem:blomer-li-pointwise}; rather, it changes the order in which the nonzero Poisson frequencies and the auxiliary variables are assembled before the quadratic estimate is applied.

For the treatment of $S_1$ (see Section \ref{sec-S1-BL}),  we further define
\begin{align*}	
T(h,d_1,d_2,n_1,n_2,p_1,p_2,t)=\sum_{x\bmod [p_1,p_2]t}e_{[p_1,p_2]t}(xn_2)S_{h,d_1,n_1,x}(p_1t)\overline{S_{h,d_2,n_1,x}(p_2t)},
\end{align*}
for two primes  $p_1, p_2$ and $n_1|t$. The following lemma is due to  \cite[Lemma 8]{Blomer-Li-2025}.

\begin{lemma}\label{lem:blomer-li-correlation}
Let $p_1, p_2$ be two primes, $n_1|t=t_{\mathrm{sf}}t_{\mathrm{sq}}$ with $t_{\mathrm{sf}}$ squarefree, $t_{\mathrm{sq}}$ squarefull and $(t_{\mathrm{sf}},t_{\mathrm{sq}})=1$.	Assume $(hn_1,p_1p_2)=1$. Then we have
\begin{align*}
p_1\ne p_2:\quad T&\ll p_1^{3/2}p_2^{3/2}t^{5/2+\ee}(hn_2,t_{\mathrm{sf}})^{1/2}t_{\mathrm{sq}}^{1/2},\\
p_1=p_2:\quad T&\ll p_1^3t^{5/2+\ee}(hn_2p_1,t_{\mathrm{sf}})^{1/2}t_{\mathrm{sq}}^{1/2},\\
p_1=p_2,\ n_2=0:\quad T&\ll p_1^2t^{3+\ee}(d_1-d_2,p_1).
\end{align*}
If $p_1\ne p_2$, then $T$ vanishes unless $(n_2,p_1p_2)=1$.
\end{lemma}

\subsection{Matrix structure of the complete character sum}\label{matrix structure}

We now isolate the matrix structure of the complete character sum. Lemma \ref{lem:fourier-representation} gives an exact Fourier representation, while Lemma 2.10 identifies the associated finite operator and provides an operator norm estimate valid uniformly for all moduli. Fixing $p$ and $l$ at that stage is unnecessarily early. In the dominant major-arc branch, all three variables $p,l,\nu$ enter the character sum only through its row parameter. Section \ref{sec-S2-BL} shows that, after writing $k=l\nu$, the relevant pairs $(p,k)$ occupy distinct rows. This permits a direct application of Lemma \ref{lem:uniform-operator-bound} to the packed coefficient vector. The purpose of the present subsection is therefore to isolate the uniform operator bound that will ultimately be applied to the entire packed family before the auxiliary variables are separated.

We begin with a simple Fourier representation of the character sum introduced above, which will be useful in the subsequent analysis.

\begin{lemma}\label{lem:fourier-representation}
For $n_1\mid c$ and $q=c/n_1$, we have
\begin{align*}
S_{h,d,n_1,n_2}(c)=\sideset{}{^{*}}\sum_{x\bmod q}S(h,d+n_1x;c)e_q(n_2\bar x).
\end{align*}
\end{lemma}

\begin{proof}
Opening the inner Kloosterman sum in \eqref{eq:basic-character-sum}, we obtain
\begin{align*}
S(\bar b,n_2;q)=\sideset{}{^{*}}\sum_{x\bmod q}e_q(\bar bx+n_2\bar x).
\end{align*}
Since the inverse of $b$ modulo $c$ reduces to the inverse of $b$ modulo $q$, and $c=n_1q$, we have	$e_q(\bar bx)=e_c(n_1\bar bx)$. Interchanging the $b$- and $x$-sums gives
\begin{align*}
\sideset{}{^{*}}\sum_{x\bmod q}e_q(n_2\bar x)\sideset{}{^{*}}\sum_{b\bmod c}e_c\Big(hb+\bar b(d+n_1x)\Big),
\end{align*}
and the inner sum is $S(h,d+n_1x;c)$.
\end{proof}

The preceding representation suggests viewing the character sum as a finite matrix kernel. We therefore introduce the following normalization
\begin{align*}
K_{n_1,c}(d,n_2)=\frac{\sqrt n_1}{c}S_{h,d,n_1,n_2}(c).
\end{align*}
The resulting matrix satisfies a uniform $\ell^2$ operator bound.

\begin{lemma}\label{lem:uniform-operator-bound}
Let $T_{n_1,c}:\ell^2(\Z/q\Z)\to\ell^2(\Z/c\Z)$ be the matrix with kernel $K_{n_1,c}(d,n_2)$. Then, we have
\begin{align*}
\|T_{n_1,c}\|_{\op}\le\sqrt c.
\end{align*}
The bound is uniform in $h$ and requires no squarefreeness assumption. Equivalently, for arbitrary complex coefficients $\gamma_d$, we have
\begin{align*}
\sum_{n_2\bmod q}\Big|\sum_{d\bmod c}\gamma_d S_{h,d,n_1,n_2}(c)\Big|^2\le \frac{c^3}{n_1}\sum_{d\bmod c}|\gamma_d|^2.
\end{align*}
\end{lemma}

\begin{proof}
For $m\geq 1$, let $F_m$ denote the unnormalized Fourier transform on $\Z/m\Z$, defined by
\begin{align*}
(F_mf)(y)=\sum_{x\bmod m}f(x)e_m(xy).
\end{align*}
Thus $\|F_m\|_{2\to2}=\sqrt m$. Let $R_q$ be restriction from $\Z/q\Z$ to $(\Z/q\Z)^\times$, and let $J_q$ denote inversion on $(\Z/q\Z)^\times$, namely
\begin{align*}
	(J_qf)(x)=f(\bar x).
\end{align*}
We use the analogous notation $J_c$ on $(\Z/c\Z)^\times$. Let
\begin{align*}
L_{q,c}:\ell^2(\Z/q\Z)\longrightarrow\ell^2((\Z/c\Z)^\times)
\end{align*}
be defined by
\begin{align*}
(L_{q,c}f)(b)=f(b\bmod q),\quad (b,c)=1.
\end{align*}
Since the reduction map
\begin{align*}
(\Z/c\Z)^\times\longrightarrow(\Z/q\Z)^\times
\end{align*}
is surjective, with every fibre of cardinality $\varphi(c)/\varphi(q)$, we have
\begin{align*}
\|L_{q,c}f\|_2^2=\frac{\varphi(c)}{\varphi(q)}\sideset{}{^{*}}\sum_{x\bmod q}|f(x)|^2\leq\frac{\varphi(c)}{\varphi(q)}\|f\|_2^2.
\end{align*}
Consequently, we derive
\begin{align*}
\|L_{q,c}\|_{2\to2}=\Big(\frac{\varphi(c)}{\varphi(q)}\Big)^{1/2}\leq\sqrt{n_1},
\end{align*}
where
\begin{align*}
\frac{\varphi(c)}{\varphi(q)}=\frac{\varphi(n_1q)}{\varphi(q)}=n_1\prod_{\substack{p\mid n_1\\ p\nmid q}}\Big(1-\frac1p\Big)\leq n_1.
\end{align*}
Whenever $F_m$ is applied to a function defined on $(\Z/m\Z)^\times$, we extend that function by zero to $\Z/m\Z$.

Let $M_h$ denote multiplication on $(\Z/c\Z)^\times$ by the phase $e_c(hb)$, so that
\begin{align*}
(M_hf)(b)=e_c(hb)f(b).
\end{align*}
The operators $R_q$ have norm at most $1$, while $J_q$, $J_c$ and $M_h$ are isometries. We now verify the operator factorization. For $f\in\ell^2(\Z/q\Z)$, successive application of the above operators gives
\begin{align*}
\big(F_cJ_cM_hJ_cL_{q,c}F_qJ_qR_qF_qf\big)(d)&=\sideset{}{^{*}}\sum_{b\bmod c}e_c(db+h\bar b)\sideset{}{^{*}}\sum_{x\bmod q}e_q(bx)\sum_{n_2\bmod q}f(n_2)e_q(n_2\bar x)\\
&=\sum_{n_2\bmod q}f(n_2)\sideset{}{^{*}}\sum_{x\bmod q}e_q(n_2\bar x)
\sideset{}{^{*}}\sum_{b\bmod c}e_c\big((d+n_1x)b+h\bar b\big).
\end{align*}
By the symmetry of Kloosterman sums, we have
\begin{align*}
S(u,v;c)=S(v,u;c),
\end{align*}
the innermost sum is $S(h,d+n_1x;c)$. Hence, we obtain
\begin{align*}
\big(F_cJ_cM_hJ_cL_{q,c}F_qJ_qR_qF_qf\big)(d)=\sum_{n_2\bmod q}S_{h,d,n_1,n_2}(c)f(n_2).
\end{align*}
Thus the matrix with kernel $S_{h,d,n_1,n_2}(c)$ is represented by
\begin{align*}
F_cJ_cM_hJ_cL_{q,c}F_qJ_qR_qF_q.
\end{align*}
It follows that its operator norm is bounded by
\begin{align*}
\sqrt c\,\Big(\frac{\varphi(c)}{\varphi(q)}\Big)^{1/2}\sqrt q\,\sqrt q\leq q\sqrt{n_1c}.
\end{align*}
Recalling that
\begin{align*}
K_{n_1,c}(d,n_2)=\frac{\sqrt n_1}{c}S_{h,d,n_1,n_2}(c),
\end{align*}
we therefore obtain
\begin{align*}
\|T_{n_1,c}\|_{\op}\leq\frac{\sqrt n_1}{c}\,q\sqrt{n_1c}=\frac{n_1q}{\sqrt c}=\sqrt c.
\end{align*}

Finally, for arbitrary complex coefficients $\gamma_d$, the adjoint operator satisfies
\begin{align*}
T_{n_1,c}^{*}(\overline{\gamma})(n_2)=\sum_{d\bmod c}\overline{K_{n_1,c}(d,n_2)}\,\overline{\gamma_d}=\overline{\sum_{d\bmod c}\gamma_dK_{n_1,c}(d,n_2)}.
\end{align*}
Hence, we obtain
\begin{align*}
\sum_{n_2\bmod q}\Big|\sum_{d\bmod c}\gamma_dK_{n_1,c}(d,n_2)\Big|^2&=\|T_{n_1,c}^{*}(\overline{\gamma})\|_2^2\leq c\sum_{d\bmod c}|\gamma_d|^2.
\end{align*}
Undoing the normalization in $K_{n_1,c}(d,n_2)$ gives
\begin{align*}
\sum_{n_2\bmod q}\Big|\sum_{d\bmod c}\gamma_dS_{h,d,n_1,n_2}(c)\Big|^2\leq\frac{c^3}{n_1}\sum_{d\bmod c}|\gamma_d|^2.
\end{align*}
The estimate is uniform in $h$ and requires no squarefreeness assumption on $c$.
\end{proof}

\section{Proof of Theorem \ref{B-L}}

In this section, we present the proof of Theorem \ref{B-L}. Motivated by  Blomer and Li \cite{Blomer-Li-2025}, we decompose the shifted convolution sum by means of Jutila's circle method and then treat the two resulting contributions separately. The treatment of $S_1$ follows their argument, whereas the major-arc analysis of $S_2$ combines the nonzero-frequency matrix estimate with cancellation of the zero-frequency main term.

We recall that the main object of interest is
\begin{align*}
\sum_{ m- n=h} \lambda_\pi(1,n)W({n})d(m)W_0({m}),
\end{align*}
where $W$ and $W_0$ are smooth functions compactly supported on $[X,2X]$.
Insert the identity $1=\chi(\alpha)+(1-\chi(\alpha))$ into the additive orthogonality representation of $m-n=h$.
By  Lemma \ref{Jutila circle method}, we decompose this quantity into two parts, denoted by
\begin{align*}
S_1=\sum_{n,m} \lambda_\pi(1,n)W({n})d(m)W_0({m})
\int_{0}^{1}\chi(\alpha)e\big(( m- n-h)\alpha\big)d\alpha,
\end{align*}
and
\begin{align*}
S_2=\sum_{n,m} \lambda_\pi(1,n)W({n})d(m)W_0({m})
\int_{0}^{1}\big(1-\chi(\alpha)\big)e\big(( m- n-h)\alpha\big)d\alpha.
\end{align*}
Applying Lemma \ref{Kloosterman circle method} along with \eqref{eq:fourier-chi}, we may rewrite $S_2$ as
\begin{align*}
S_2=&\sum_{n,m} \lambda_\pi(1,n)W({n})d(m)W_0({m})\\
&\times\sum_{c\le C}\;\sideset{}{^*}\sum_{\substack{b\bmod{c}}}
\int_{-\frac{1}{c(c+c')}}^{\frac{1}{c(c+c'')}}\Big(1-\chi\Big(\frac{b}{c}+z\Big)\Big)
e\Big(( m- n-h)\Big(\frac{b}{c}+z\Big)\Big)dz.
\end{align*}

We recall that $\chi$ depends on the quantities $Q, \delta, w$ (which in turn determine $L$), together with the smooth cutoff function $\psi$. The function $\psi$ remains fixed once and for all for the remainder of our argument. We choose a factorable Jutila weight
\begin{equation*}
	w(q)=\sum_{\substack{Q_1/2\le p\le Q_1\\p\ \mathrm{prime},\ p\nmid h}}
	\sum_{pt=q}\rho(t/Q_2),
\end{equation*}
where $Q=Q_1Q_2$, and $\rho$ is a fixed smooth nonnegative, nonzero function supported in $[1/2,1]$. In the parameter range used in the endgame, $Q_1$ is a fixed positive power of $X$. The exclusion $p\mid h$ removes only $O(1)$ primes from $[Q_1/2,Q_1]$, whereas the prime number theorem supplies $\gg Q_1/\log Q_1$ primes there. Since $pt\asymp Q$ on the support of $\rho$ and $\varphi(pt)\gg pt/\log\log(3pt)$, we have
\begin{equation*}
Q^2X^{-o(1)}\ll L=\sum_q\varphi(q)w(q)\ll Q^2X^{o(1)},\quad \|w\|_\infty\ll X^\ee.
\end{equation*}
Thus throughout the proof we may write $L=Q^{2+o(1)}$; in particular $(\delta L)^{-1}=X^\ee(\delta Q^2)^{-1}$. In order to simplify our notation throughout the subsequent analysis, we impose the following working hypotheses
$\log Q\asymp|\log\delta| \asymp \log X.$

We next treat the contributions arising from $S_2$ and $S_1$ separately.

\subsection{Estimation of  $S_2$}  \label{sec-S2-BL}

We choose parameters $C$ and $C_0$ satisfying $C_0<C$ and $\log C_0\asymp\log C\asymp \log X$.
The decomposition $c\le C_0$ and $C_0< c\le C$ separates the contribution of $S_2$ into the major-arc part $S_{2,0}$ and the minor-arc part $S_{2,1}$, respectively.

The analysis of the minor-arc part $S_{2,1}$ proceeds as in Blomer and Li \cite[Section 3.1]{Blomer-Li-2025}. By Cauchy–Schwarz over the disjoint Farey arcs and Lemma \ref{Jutila circle method},  we get
\begin{align*}
|S_{2,1}|\ll& \Big(\int_0^1 |1-\chi(z)|^2dz\Big)^{1/2}\Big(\sum_n|\lambda_\pi(1,n)W(n)|^2\Big)^{1/2}\\
\times&\max_{C_0<c\le C}\max_{\substack{b\bmod c\\(b,c)=1}}\max_{|z|\le(cC)^{-1}}
\Big|\sum_m d(m)W_0(m)e\Big(\Big(\f{b}{c}+z\Big)m\Big)\Big|.
\end{align*}
For the first integral, we invoke Lemma \ref{Jutila circle method} to derive the upper bound $X^\ee{Q^2}/{L^2\delta}$.  By Lemma \ref{Voronoi-d(n)}, we derive
\begin{align*}
\sum_m d(m)W_0(m)e\Big(\Big(\f{b}{c}+z\Big)m\Big)\ll X^\ee\Big(\frac Xc+c(1+X|z|)\Big)
\ll X^\ee\Big(\frac X{C_0}+C\Big).
\end{align*}
By  Rankin--Selberg theory, the sum over $n$ is bounded by $X$.
Since $\delta=X^{-1+\ee_0}$, we get
\begin{align}\label{S_21-bound-BL}
S_{2,1}\ll X^\ee\frac{X^{1/2}}{Q\delta^{1/2}}\Big(\frac X{C_0}+C\Big)\ll X^\ee\Big(\frac{X^2}{QC_0}+\frac{XC}{Q}\Big).
\end{align}

We now turn to bound the major arc contribution $S_{2,0}$ with $c\le C_0$. The main saving arises from the treatment of $S_{2,0}$. For $c\le C_0$ use the refined form \eqref{eq:refined-farey}. Apply the divisor Voronoi formula to the $m$-sum. Assume
\begin{align}\label{eq:major-arc-range}
C_0\le X^{1/2-\eta},\qquad C\ge X^{1/2+\eta}
\end{align}
for some fixed $\eta>0$. Then the dual divisor sum is negligible by \eqref{eq:divisor-minus-transform}. Integration by parts in the main term also restricts $|z|\ll X^{-1+\ee_1}$. By \eqref{eq:farey-t-sum}, the $t$-sum in the refined Farey identity is bounded by $c/(1+|u|)$ for $|u|\le c/2$. Applying $\mathrm{GL}(3)$ Voronoi and Remark \ref{rem-Gz}, we obtain (see also the expression after (3.9) in \cite{Blomer-Li-2025}.)
\begin{align}	\label{eq:s2-major-arc-master}
\begin{aligned}
S_{2,0}\ll& X^{3/2+\ee}\sum_{c\le C_0}c^{-3/2}\int_{|z|\ll X^{-1+\ee_1}}\sum_\pm\sum_{n_2\ge1}\frac1{\sqrt {n_2}}\sum_{n_1\mid c}|\lambda_\pi(n_2,n_1)|\Big(1+\frac{n_1^2n_2X}{c^3}\Big)^{-A}\\
&\times\sum_{|u|\le c/2}\frac{\mathcal C(c,n_1,n_2,u;z)}{1+|u|}\,dz,
\end{aligned}
\end{align}
where
\begin{align}	\label{eq:major-arc-character-sum}
\mathcal C(c,n_1,n_2,u;z)=\Big|\;\sideset{}{^*}\sum_{b\bmod c}e_c(u\bar b-hb)S(-\bar b,\pm n_2;c/n_1)\Big(1-\chi\Big(\frac bc+z\Big)\Big)\Big|.
\end{align}

For the $\chi$-part, we insert \eqref{chi-repr}. Thus we rewrite the $\chi$-part as
\begin{align*}
\begin{aligned}
\frac{1}{\delta L}\;\sideset{}{^*}\sum_{b\bmod c}e_c(u\bar b-hb)S(-\bar b,\pm n_2;c/n_1)\sum_{q} w(q)\sum_{d\mid q}\mu\Big(\frac qd\Big)
\sum_{l\equiv bd\bmod c} \psi\Big(\frac1\delta\Big(\frac{l}{cd}+z\Big)\Big).
\end{aligned}
\end{align*}
When $l=0$, the congruence $l\equiv bd\bmod c$ forces $c\mid d$ and hence $c\mid q$. In this case the smooth factor is $\psi(z/\delta)$, independent of $d$, and M\"obius inversion gives the exact identity
\begin{align*}
\sum_{\substack{d\mid q\\c\mid d}}\mu(q/d)=
\begin{cases}
1,&q=c,\\
0,&q\ne c.
\end{cases}
\end{align*}
Assuming $4C_0<Q$, since $w$ is supported on $[Q/4,Q]$, the multiplier associated to the circle frequency $l=0$ is
\begin{align*}	
\frac{\psi(z/\delta)}{\delta L}\sum_q w(q)\sum_{\substack{d\mid q\\c\mid d}}\mu(q/d)=\frac{\psi(z/\delta)w(c)}{\delta L}=0.
\end{align*}

We therefore consider $l\ne0$, retaining the original signed sums until the Poisson frequencies have been separated. The contribution of $1$ inside the complete sum will be kept together with the zero Poisson frequency of $-\chi$. 

Let $g_1=(l,c)=(d,c)$, $c=g_1g_2$, and divide $l$ and $d$ by $g_1$. Then $(l,g_2)=(d,g_2)=1$ and the relevant part is
\begin{align*}
\begin{aligned}	
&\frac1{\delta L}\;\sideset{}{^{*}}\sum_{b\bmod c}e_c(u\bar b-hb)S(-\bar b,\pm n_2;c/n_1)\\
&\times	\sum_{g_1g_2=c}\sum_{\substack{l\ne0\\(l,g_2)=1}}\sum_r\mu(r)
\sum_{\substack{(d,g_2)=1\\d\equiv\bar bl\, \bmod\,g_2}}
w(dg_1r)\psi\Big(\delta^{-1}\Big(\frac l{cd}+z\Big)\Big).
\end{aligned}
\end{align*}
From the definition of $w$, it follows that the equation $pt=dg_1r$ is decomposed exactly as in \cite[(3.13)--(3.15)]{Blomer-Li-2025}. After the same changes of variables, we use the following three branches:
\begin{align*}
B_1:&\quad g_1g_2=c,\quad (p,r)=1,\quad d\equiv\bar bl\bmod{g_2},\\
B_2:&\quad pg_1g_2=c,\quad (p,r)=1,\quad d\equiv\bar bl\bmod{g_2},\\
B_3:&\quad g_1g_2=c,\quad (p,r)=1,\quad (p,c)=1,\quad d\equiv\bar p\bar bl\bmod{g_2}.
\end{align*}
In all three branches, the smooth factor is $\rho(dg_1r/Q_2)$. The third branch  has $	\psi(\delta^{-1}( l/{(cdp)}+z))$, while the first two have $\psi(\delta^{-1}(l/(cd)+z))$. After a dyadic partition $d\asymp Y$, one  has $1\le Y\ll{Q_2}/{(g_1r)}.$
The support of $\psi$ gives
\begin{align*}	
|l|\ll cY\varpi\delta,\quad\varpi=1\ \text{in }B_1,B_2,\qquad \varpi=p\ \text{in }B_3.
\end{align*}
Indeed $|z|/\delta\ll X^{\ee_1-\ee_0}=o(1)$, and if a nonzero integer $l$ lies in the support then	${|l|}/{(cY\varpi)}\ll\delta+|z|\ll\delta$, 
which itself implies $cY\varpi\delta\gg1$.

The first two branches will disappear in the final parameter range. More precisely, we impose
\begin{align}\label{eq:eliminate-first-two-branches}
C_0Q_2\delta=o(1).
\end{align}
Since $cY\delta\ll C_0Q_2\delta$ in $B_1$ and $B_2$, no nonzero integer $l$ can occur there for sufficiently large $X$. Thus the entire nonzero-$l$ analysis reduces to $B_3$.

We only need to consider Poisson summation in the third branch
\begin{align*}
\begin{aligned}	
&\frac1{\delta L}\;\sideset{}{^{*}}\sum_{b\bmod c}e_c(u\bar b-hb)S(-\bar b,\pm n_2;c/n_1)\\
&\times \sum_{\substack{Q_1/2\leq p\leq Q_1\\p {\text{ prime}}\\ p\nmid ch}}	\sum_{g_1g_2=c}\sum_{\substack{l\ne0\\(l,g_2)=1}}\sum_{(p,r)=1}\mu(r)\sum_{\substack{(d,g_2)=1\\d\equiv\overline{pb}l\, \bmod\,g_2}}\rho\Big(\f{dg_1r}{Q_2}\Big)\psi\Big(\delta^{-1}\Big(\frac l{cdp}+z\Big)\Big).
\end{aligned}
\end{align*}
Poisson summation modulo $g_2$ gives the weight
\begin{align}	\label{eq:poisson-weight}
\alpha_{p,l,\nu}=\frac1{g_2}\int_{\R}v(\xi/Y)\rho(\xi g_1r/Q_2)\psi\Big(\delta^{-1}\Big(\frac l{c\xi p}+z\Big)\Big)e\Big(-\frac{\nu\xi}{g_2}\Big)\,d\xi.
\end{align}
To avoid confusing the original $d$-variable with its Poisson dual, the latter is denoted by $\nu$. 
Repeated integration by parts gives, for every $A>0$,
\begin{align*}
	|\alpha_{p,l,\nu}|\ll_A\frac Y{g_2}\Big(1+\frac{|\nu|Y}{g_2}\Big)^{-A}.
\end{align*}
Put $H={g_2}/{Y}.$  For a fixed sufficiently small $\eta_0>0$, the contribution of $|\nu|>HX^{\eta_0}$ is $O(X^{-A})$. If $HX^{\eta_0}<1$, every nonzero frequency is already negligible. Otherwise the relevant nonzero frequencies lie in two intervals of  $O(HX^{\eta_0})$. Moreover, we get
\begin{align}	\label{eq:poisson-weight-energy}
\sum_{\nu\ne0}|\alpha_{p,l,\nu}|^2\ll_A H^{-1}.
\end{align}

We next track the $b$-phase exactly through the same changes of variables used in the branch decomposition. It is convenient to write $\varpi=1$ in $B_1$ and $\varpi=p$ in $B_3$. The Poisson calculation following the transformed congruence gives
\begin{align*}
\sideset{}{^{*}}\sum_{b\bmod c}e_c\bigl(-hb+\bar b\bar\varpi(l\nu g_1+\varpi u)\bigr)S(-\bar b,\pm n_2;c/n_1).
\end{align*}
Thus in the third branch, where $\varpi=p$ and $p\nmid c$, the $b$-sum is exactly
\begin{align}\label{eq:exact-poisson-phase}
\sideset{}{^{*}}\sum_{b\bmod c}e_c\bigl(-hb+\bar b(u+g_1 l\nu\bar p)\bigr)S(-\bar b,\pm n_2;c/n_1).
\end{align}
The inverse $\bar p$ is therefore part of the exact Poisson phase, which is not introduced by an estimate. This point is important for the row map below. After $b\mapsto-b$ and the corresponding sign changes in the Kloosterman sum, \eqref{eq:exact-poisson-phase} is
$S_{h,-u-g_1 l\nu\bar p,n_1,\pm n_2}(c).$
What matters here is that $p,l,\nu$ occur only in the row variable $d$; the $\mathrm{GL}(3)$ dual variable $n_2$, the divisor $n_1\mid c$, and the normalized Voronoi weight are common to the whole family.

We retain these variables through the quadratic step rather than fixing $p$ and $l$ first. Set
\begin{align}\label{eq:packed-coefficients}
\beta_{p,k}=\sum_{\substack{l\nu=k\\ l\ne0,\ \nu\ne0}}\alpha_{p,l,\nu},
\end{align}
where $l$ and $\nu$ are restricted to the actual supports in $B_3$ and $0<|\nu|\ll HX^{\eta_0}$. Then the $n_2$-dependent expression in $B_3$ becomes
\begin{align*}
T(n_2)=\sum_{p\asymp Q_1}\sum_{k\ne0}\beta_{p,k}S_{h,d_0-g_1k\bar p,n_1,\pm n_2}(c),
\end{align*}
where $d_0=-u$. In particular, we have $|k|=|l\nu|\ll cg_2Q_1\delta X^{\eta_0}.$

\begin{lemma}\label{lem:packed-row-injectivity}
Assume  $C_0^2\delta X^{\eta_0}=o(1),  C_0Q_1^2\delta X^{\eta_0}=o(1).$ Then, for $c\le C_0$ and $X$ sufficiently large, the map	$(p,k)\longmapsto d_0-g_1k\bar p\bmod c$ is injective on the support of $T(n_2)$.
\end{lemma}

\begin{proof}
Let	$K^*=cg_2Q_1\delta X^{\eta_0}.$ 
Since $g_2\le c\le C_0$, the first condition $C_0^2\delta X^{\eta_0}=o(1)$ 
gives $K^*<Q_1/2$, while the second gives $2Q_1K^*<g_2$, once $X$ is sufficiently large. Suppose two pairs $(p_1,k_1)$ and $(p_2,k_2)$ have the same image. Because $p_i\nmid c$ in $B_3$,
\begin{align*}
g_1k_1\bar p_1\equiv g_1k_2\bar p_2\pmod{g_1g_2}
\end{align*}
implies
\begin{align*}
k_1p_2-k_2p_1\equiv0\pmod{g_2}.
\end{align*}
On the other hand, we have
\begin{align*}
|k_1p_2-k_2p_1|\le2K^*Q_1<g_2.
\end{align*}
Thus, we derive  $k_1p_2=k_2p_1.$ If $p_1\ne p_2$, primality implies $p_1\mid k_1$. But $0<|k_1|<Q_1/2\le p_1$, this is a contradiction. Therefore $p_1=p_2$, and then $k_1=k_2$.
\end{proof}

The next estimate is the point at which the counting of $p$ and $l$ changes character. If these variables were summed outside the quadratic estimate, their cardinalities would enter linearly. After packing, they contribute to the energy of a single coefficient vector.

\begin{lemma}\label{lem:packed-coefficient-energy}
For fixed $c,g_1,g_2,r,Y,n_1,u,z$ in $B_3$, we have
\begin{align*}
\sum_{p,k}|\beta_{p,k}|^2\ll X^\ee Q_1^2g_1Y^2\delta,\quad\|\beta\|_2\ll X^\ee Q_1Y\sqrt{g_1\delta}.
\end{align*}
\end{lemma}

\begin{proof}
For a fixed nonzero integer $k$, the number of factorizations $k=l\nu$ is $d(|k|)\ll X^\ee$ throughout the polynomial support under consideration. Cauchy--Schwarz in the divisor relation therefore gives
\begin{align*}
|\beta_{p,k}|^2\ll X^\ee\sum_{\substack{l\nu=k\\l,\nu\ne0}}|\alpha_{p,l,\nu}|^2.
\end{align*}
Summing over $p$ and $k$, we obtain
\begin{align*}
\sum_{p,k}|\beta_{p,k}|^2\ll X^\ee\sum_{p,l}\sum_{\nu\ne0}|\alpha_{p,l,\nu}|^2.
\end{align*}
There are $O(Q_1)$ admissible primes and $O(cYQ_1\delta)$ admissible nonzero $l$. By \eqref{eq:poisson-weight-energy}, the $\nu$-sum contributes $H^{-1}=Y/g_2$. Hence, for $c=g_1g_2$, we have
\begin{align*}
\sum_{p,k}|\beta_{p,k}|^2\ll X^\ee Q_1(cYQ_1\delta)\frac Y{g_2}=X^\ee Q_1^2g_1Y^2\delta.
\end{align*}
This completes the proof of Lemma \ref{lem:packed-coefficient-energy}.
\end{proof}

The $\mathrm{GL}(3)$ Voronoi weight restricts a dyadic block $n_2\asymp N$ to $N\ll X^\ee{c^3}{(n_1^2X)^{-1}}.$ Since $q=c/n_1$, the major-arc condition $c\le C_0\le X^{1/2-\eta}$ gives $N/q\ll X^{-2\eta+\ee}<1$. 
Thus the interval $n_2\asymp N$ contains at most one representative of every residue class $\bmod q$. By Lemma \ref{lem:packed-row-injectivity}, the coefficients $\beta_{p,k}$ occupy distinct rows modulo $c$. Therefore the dual form of Lemma \ref{lem:uniform-operator-bound} gives
\begin{align*}
\sum_{n_2\bmod q}|T(n_2)|^2\le\frac{c^3}{n_1}\sum_{p,k}|\beta_{p,k}|^2.
\end{align*}
Let $U(n_2)$ denote the remaining normalized $\mathrm{GL}(3)$ Voronoi cutoff, so $|U(n_2)|\ll1$. By Cauchy--Schwarz and \eqref{lem:hecke-moments}, we derive
\begin{align*}
\begin{aligned}
\sum_{n_2\asymp N}\frac{|\lambda_\pi(n_2,n_1)|}{\sqrt n_2}|U(n_2)||T(n_2)|&\le\Big(\sum_{n_2\asymp N}\frac{|\lambda_\pi(n_2,n_1)|^2}{n_2}|U(n_2)|^2\Big)^{1/2}\Big(\sum_{n_2\asymp N}|T(n_2)|^2\Big)^{1/2}\\
&\ll X^\ee n_1^{1/2}\Big(\frac{c^3}{n_1}\Big)^{1/2}\|\beta\|_2=X^\ee c^{3/2}\|\beta\|_2.
\end{aligned}
\end{align*}
Combining  this and Lemma \ref{lem:packed-coefficient-energy},  the total nonzero Poisson-frequency contribution of the third branch is
\begin{align}	\label{eq:nonzero-poisson-total}
\begin{aligned}	
S_{2,0}^{B_3;\,l\ne0,\nu\ne0}&\ll\f{X^{1/2+\ee}}{\delta L}\sum_{c\leq C_0}c^{-\f{3}{2}}\sum_{g_1g_2=c}\sum_{r}\sum_{n_2\asymp N}\frac{|\lambda_\pi(n_2,n_1)|}{\sqrt n_2}|U(n_2)||T(n_2)|\\
&\ll \frac{X^{1/2+\ee}}{\delta Q^2}\sum_{c\le C_0}\sum_{g_1g_2=c}\sum_{r\le Q_2}Q_1Y\sqrt{g_1\delta}\\
&\ll \frac{X^{1+\ee}}{Q}\sum_{c\leq C_0}\sum_{g_1|c}\sum_{r\leq Q_2}\f{1}{\sqrt{g_1}r}\ll\frac{X^{1+\ee}C_0}{Q}.
\end{aligned}
\end{align}

Next, we combine the zero Poisson frequency ($\nu=0$) with the contribution of $1$ in \eqref{eq:major-arc-character-sum} before estimating the resulting character sum. In addition to \eqref{eq:major-arc-range}, \eqref{eq:eliminate-first-two-branches} and Lemma \ref{lem:packed-row-injectivity}, assume that $4C_0<Q, C_0Q_1\delta=o(1).$ We write
\begin{align*}
\mathcal P&=\{p:Q_1/2\le p\le Q_1,\ p\ \mathrm{prime},\ p\nmid h\},
\end{align*}
and
\begin{align*}
I_\rho&=\int_0^\infty t\rho(t/Q_2)\,dt\asymp Q_2^2.
\end{align*}

Let $\chi_0(c,z)$ denote its zero Poisson frequency ($\nu=0$) in \eqref{eq:poisson-weight}, summed over all the dyadic localizations in $d$. This quantity is independent of $b$.
Using \eqref{eq:poisson-weight} with $\nu=0$ and summing over the dyadic partition. Before reindexing, the summands are
\begin{align*}	
\chi_0(c,z)=\frac1{\delta L}\sum_{\substack{p\in\mathcal P\\p\nmid c}}\sum_{g_1g_2=c}\sum_{\substack{r\ge1\\(r,p)=1}}\mu(r)\sum_{\substack{l\ne0\\(l,g_2)=1}}\frac1{g_2}\int_0^\infty \rho(\xi g_1r/Q_2)\psi\Big(\frac l{cp\xi\delta}+\frac z\delta\Big)\,d\xi.
\end{align*}
There is no endpoint contribution from the original restriction $d\ge1$: on the support of the integrand, $l\ne0$ and $|z|/\delta=o(1)$ imply
\begin{align*}
\xi\gg\frac1{cp\delta}\ge\frac1{C_0Q_1\delta}\longrightarrow\infty.
\end{align*}
Hence the dyadic partition sums to $1$ throughout this support. Set $t=g_1r\xi$, so that $g_2^{-1}d\xi=(cr)^{-1}dt$, and then put $m=g_1 l$. The correspondence
\begin{align*}
(g_1,l)\longleftrightarrow m,\quad g_1=(m,c),\quad (l,c/g_1)=1,
\end{align*}
is a bijection between the indicated pairs and the nonzero integers $m$. Finally set $j=mr$ and sum over $r\mid|j|$. The support gives
\begin{align*}
|j|\ll cQ\delta,\quad g_1r\le|j|,\quad\xi\asymp\frac{Q_2}{g_1r}\gg\frac1{cQ_1\delta}\gg1.
\end{align*}
In particular, every divisor introduced by this reindexing lies in the original $r$-range.  Thus, we obtain the exact identity
\begin{align}	\label{eq:zero-poisson-exact}
\chi_0(c,z)&=\frac1{\delta Lc}\sum_{\substack{p\in\mathcal P\\p\nmid c}}\int_0^\infty\rho(t/Q_2)
\sum_{j\in\Z\setminus\{0\}}a_p(|j|)\psi\Big(\frac{j}{cpt\delta}+\frac z\delta\Big)\,dt,
\end{align}
where
\begin{align*}
a_p(j)&=\sum_{\substack{r\mid j\\(r,p)=1}}\frac{\mu(r)}r.
\end{align*}

At this point, the zero Poisson frequency has been reduced to the one-dimensional arithmetic average, with all dependence on the reduced residue class removed. It remains to evaluate the weighted mean of the coefficients $a_p(|j|)$ appearing there. The following lemma identifies its main term uniformly in the scale and translation parameters, and hence gives the required asymptotic for $\chi_0(c,z)$.
\begin{lemma}\label{eq:zero-poisson-main-term}
With the notations as above, we have
\begin{align*}
\chi_0(c,z)=\frac{I_\rho}{\zeta(2)L}\sum_{\substack{p\in\mathcal P\\p\nmid c}}\frac{p}{1-p^{-2}}+O\Big(\frac{X^\ee}{cQ\delta}\Big).
\end{align*}
\end{lemma}
\begin{proof}
Assume $Z> 0$ and $|k|\leq 1/4$. Opening the divisor sum and writing $j=rm$, we obtain
\begin{align*}
\sum_{j\ne0}a_p(|j|)\psi(j/Z+k)&=\sum_{\substack{r\geq1\\(r,p)=1}}\frac{\mu(r)}r\sum_{m\ne0}\psi\Big(\frac{rm}{Z}+k\Big).
\end{align*}
Since $\psi$ is supported on $[-1,1]$ and $|k|\leq1/4$, a nonzero term in the inner sum satisfies
\begin{align*}
\Big|\frac{rm}{Z}\Big|\leq	1+|k|\leq\frac54.
\end{align*}
As $m\ne0$, this implies $r\leq 5Z/4$, and in particular $r\leq2Z$. Hence
\begin{align*}
\sum_{j\ne0}a_p(|j|)\psi(j/Z+k)&=\sum_{\substack{r\leq2Z\\(r,p)=1}}\frac{\mu(r)}r\sum_{m\ne0}\psi\Big(\frac{rm}{Z}+k\Big).
\end{align*}
We next estimate the inner sum. Put $t=r/Z$. By the elementary summation formula for a compactly supported smooth function, we have
\begin{align*}
\sum_{m\in\mathbb Z}\psi(tm+k)=\frac1t\int_{\mathbb R}\psi(y)\,dy+O_\psi(1),
\end{align*}
uniformly for $t>0$ and $k\in\mathbb R$. Indeed, if $f(x)=\psi(tx+k)$, then
\begin{align*}
\int_{\mathbb R}f(x)\,dx=\frac1t\int_{\mathbb R}\psi(y)\,dy,
\end{align*}
while the error in replacing the integral by the sum over integers is bounded by
\begin{align*}
\int_{\mathbb R}|f'(x)|\,dx=\int_{\mathbb R}|\psi'(y)|\,dy\ll_\psi1.
\end{align*}
Since $\int\psi=1$, we obtain
\begin{align*}
\sum_{m\in\mathbb Z}\psi\Big(\frac{rm}{Z}+k\Big)=\frac Zr+O_\psi(1).
\end{align*}
Removing the term $m=0$ changes the left-hand side by $\psi(k)=O_\psi(1)$, and therefore
\begin{align*}
\sum_{m\ne0}\psi\Big(\frac{rm}{Z}+k\Big)=\frac Zr+O_\psi(1),
\end{align*}
uniformly for $r\leq2Z$ and $|k|\leq1/4$. Substituting this estimate gives
\begin{align*}
\sum_{j\ne0}a_p(|j|)\psi(j/Z+k)&=Z\sum_{\substack{r\leq2Z\\(r,p)=1}}\frac{\mu(r)}{r^2}+O\Big(\sum_{r\leq2Z}\frac1r\Big)\\
&=Z\sum_{\substack{r\leq2Z\\(r,p)=1}}\frac{\mu(r)}{r^2}+O(\log(2+Z)).
\end{align*}
Since the series over $r$ converges absolutely, we may extend the main sum to infinity at a cost $O(1)$. Finally, we derive
\begin{align*}
\sum_{\substack{r\geq1\\(r,p)=1}}\frac{\mu(r)}{r^2}=\prod_{l\ne p}\Big(1-\frac1{l^2}\Big)=\frac1{\zeta(2)}\Big(1-\frac1{p^2}\Big)^{-1}.
\end{align*}
Consequently, we obtain
\begin{align}\label{eq:ap-weighted-mean}
\sum_{j\ne0}a_p(|j|)\psi(j/Z+k)=\frac{Z}{\zeta(2)(1-p^{-2})}+O(\log(2+Z)),
\end{align}
uniformly in $p$ and $|k|\leq1/4$. 
Applying \eqref{eq:ap-weighted-mean} in \eqref{eq:zero-poisson-exact}, with $Z=cpt\delta$ and $k=z/\delta$, yields this lemma.
\end{proof}

To turn the preceding asymptotic into cancellation with the original constant multiplier $1$, it remains to compare its main term with the Jutila normalization $L$. We therefore evaluate $L$ to the same precision, retaining the local factor $(1-p^{-2})^{-1}$.
\begin{lemma}\label{eq:jutila-normalization-precise}
With the notations as above, we have 
\begin{align*}
1=\frac{I_\rho}{\zeta(2)L}\sum_{\substack{p\in\mathcal P}}\frac{p}{1-p^{-2}}+O(X^\ee Q_1^{-1}+X^\ee Q_2^{-1}).
\end{align*}
\end{lemma}
\begin{proof}
We first note that
\begin{align}\label{phi(pt)}
\varphi(pt)=(p-1)t\sum_{\substack{d\mid t\\(d,p)=1}}\frac{\mu(d)}d.
\end{align}
Put $f_d(x)=x\rho(dx/Q_2)$. By the summation formula, we have
\begin{align*}
\sum_{k\geq1}f_d(k)=\int_0^\infty f_d(x)\,dx+O\Big(\sup_x|f_d(x)|+\int_0^\infty|f_d'(x)|\,dx\Big).
\end{align*}
Since $\rho$ is supported in a fixed compact subinterval of $(0,\infty)$, the function
$f_d$ is supported on $x\asymp Q_2/d$. Hence, we obtain
\begin{align*}
\sup_x|f_d(x)|+\int_0^\infty|f_d'(x)|\,dx\ll\frac{Q_2}{d}.
\end{align*}
Moreover, after the change of variables $t=dx$, we get
\begin{align*}
\int_0^\infty x\rho(dx/Q_2)\,dx=\frac1{d^2}\int_0^\infty t\rho(t/Q_2)\,dt=\frac{I_\rho}{d^2}.
\end{align*}
Thus, for  $d\ll Q_2$, we have
\begin{align*}
\sum_{k\geq1}k\rho(dk/Q_2)=\frac{I_\rho}{d^2}+O\Big(\frac{Q_2}{d}\Big).
\end{align*} 
Using the preceding identity and writing $t=dk$, we obtain
\begin{align*}
L&=\sum_{p\in\mathcal P}\sum_{t\geq1}\varphi(pt)\rho(t/Q_2)=\sum_{p\in\mathcal P}(p-1)\sum_{t\geq1}t\rho(t/Q_2)\sum_{\substack{d\mid t\\(d,p)=1}}\frac{\mu(d)}d\\
&=\sum_{p\in\mathcal P}(p-1)\sum_{\substack{d\geq1\\(d,p)=1}}\mu(d)\sum_{k\geq1}k\rho(dk/Q_2)\\
&=I_\rho\sum_{p\in\mathcal P}(p-1)\sum_{\substack{d\ll Q_2\\(d,p)=1}}\frac{\mu(d)}{d^2}+O\Big(Q_2\sum_{p\in\mathcal P}p\sum_{d\ll Q_2}\frac1d\Big).
\end{align*}
Since $I_\rho\asymp Q_2^2$, extending the $d$-sum in the main term to infinity costs $Q_2$. The tail contributes $O(pQ_2)$ for each $p$, and is absorbed by the preceding error. Finally, we get
\begin{align*}
\sum_{\substack{d\geq1\\(d,p)=1}}\frac{\mu(d)}{d^2}=\prod_{l\neq p}\Big(1-\frac1{l^2}\Big)=\frac1{\zeta(2)(1-p^{-2})}.
\end{align*}
Consequently, we arrive at
\begin{align*}
1=\frac{I_\rho}{\zeta(2)L}\sum_{p\in\mathcal P}\frac{p-1}{1-p^{-2}}+O\Big(L^{-1}Q_2\log(2Q_2)\sum_{p\in\mathcal P}p\Big).
\end{align*}
The tail in the $d$-sum contributes $O(pQ_2)$ for each $p$ and is included in this error. 
Replacing $p-1$ by $p$ has relative cost $O(Q_1^{-1})$.  This completes the proof of the lemma.
\end{proof}

The only discrepancy between the prime sums in Lemmas \ref{eq:zero-poisson-main-term} and \ref{eq:jutila-normalization-precise} is the exclusion of the primes $p\mid c$ in Lemma \ref{eq:zero-poisson-main-term}. Since at most $O(\log(2c)/\log Q_1)=O(1)$ primes in $\mathcal P$ divide $c$, whereas $\#\mathcal P\gg Q_1/\log Q_1$, their total contribution is $O(X^\varepsilon Q_1^{-1})$ relative to the main term. Hence, comparing Lemmas \ref{eq:zero-poisson-main-term} and \ref{eq:jutila-normalization-precise},  we conclude that
\begin{align}\label{eq:zero-poisson-cancellation-detailed}
|1-\chi_0(c,z)|\ll X^\ee\Big(\frac1{cQ\delta}+\frac1{Q_1}+\frac1{Q_2}\Big).
\end{align}
We now return to the character sum in \eqref{eq:major-arc-character-sum}. Its contribution from $1$ and the zero Poisson frequency is exactly
\begin{align*}
(1-\chi_0(c,z))\sideset{}{^{*}}\sum_{b\bmod c}e_c(u\bar b-hb)S(-\bar b,\pm n_2;c/n_1).
\end{align*}
The multiplier is independent of $b$, so \eqref{eq:zero-poisson-cancellation-detailed} may be applied before the pointwise character-sum bound. For completeness,  we write $c=c_{\mathrm{sf}}c_{\mathrm{sq}}$, with $c_{\mathrm{sf}}$ squarefree, $c_{\mathrm{sq}}$ squarefull and $(c_{\mathrm{sf}},c_{\mathrm{sq}})=1$. Lemma \ref{lem:blomer-li-pointwise} bounds the complete sum by $X^\ee c c_{\mathrm{sq}}^{1/2}$. On a dyadic block $N\ll X^\ee c^3/(n_1^2X)$, \eqref{lem:hecke-moments} gives
\begin{align*}
\sum_{n_2\asymp N}\frac{|\lambda_\pi(n_2,n_1)|}{\sqrt n_2}\ll X^\ee n_1^{1/2}N^{1/2}\ll X^\ee\frac{c^{3/2}}{n_1^{1/2}X^{1/2}}.
\end{align*}
The common Voronoi and $z$ factor is $X^{1/2+\ee}c^{-3/2}$. After summing over $n_1\mid c$, the signs, the dyadic blocks and $u$, the fixed-$c$ bound is therefore $X^\ee c c_{\mathrm{sq}}^{1/2}$ times the right side of \eqref{eq:zero-poisson-cancellation-detailed}. 
Thus the combined contribution of $1$ and the zero Poisson frequency in \eqref{eq:s2-major-arc-master} is
\begin{align*}
\ll X^\ee\Big(\frac{C_0}{Q\delta}+\frac{C_0^2}{Q_1}+\frac{C_0^2}{Q_2}\Big)\ll X^\ee\Big(\frac{XC_0}{Q}+\frac{C_0^2}{Q_1}+\frac{C_0^2}{Q_2}\Big).
\end{align*}
Combining this with \eqref{eq:nonzero-poisson-total} proves the following estimate.

\begin{prop}\label{prop:s2-major-arcs}
Assume \eqref{eq:major-arc-range}, \eqref{eq:eliminate-first-two-branches}, and suppose that $4C_0<Q$, $C_0Q_1\delta=o(1)$, $C_0^2\delta X^{\eta_0}=o(1)$ and   $C_0Q_1^2\delta X^{\eta_0}=o(1)$. Then we have
\begin{align*}
S_{2,0}\ll X^\ee\Big(\frac{XC_0}{Q}+\frac{C_0^2}{Q_1}+\frac{C_0^2}{Q_2}\Big).
\end{align*}
Consequently, together with \eqref{S_21-bound-BL}, we have
\begin{align*}
S_2\ll X^\ee\biggl(\frac{X^2}{QC_0}+\frac{XC}{Q}+\frac{XC_0}{Q}+\frac{C_0^2}{Q_1}+\frac{C_0^2}{Q_2}\biggr).
\end{align*}
\end{prop}

\subsection{Estimation of  $S_1$} \label{sec-S1-BL}

We now turn to the contribution from $S_1$. In this subsection we follow the  argument of Blomer and Li \cite[Section 3.2]{Blomer-Li-2025}, specialized to $\lambda_1=\lambda_2=1$. We include the details needed below. 

From Lemma \ref{Jutila circle method}, we have
\begin{align*}	
\begin{aligned}
S_1=&\frac1{\delta L}\sum_{q\le Q}w(q)\sideset{}{^{*}}\sum_{a\bmod q}\int_{\R}\psi(z/\delta)e(-(a/q+z)h)\\
&\times\sum_m d(m)W_0(m)e\Big(\Big(\f aq+z\Big)m\Big)\sum_n\lambda_\pi(1,n)W(n)e\Big(-\Big(\f aq+z\Big)n\Big)\,dz.
\end{aligned}
\end{align*}

Apply Lemma \ref{Voronoi-d(n)} to the $m$-sum and Lemma \ref{voronoi3} to the $n$-sum. With $\delta=X^{-1+\ee_0}$, the factors involving $z$ are slowly varying on the supports of the fixed weights. Place $n_1\asymp H$ and $n_2\asymp N$. By Remark \ref{rem-Gz}, up to a negligible error, we get $H^2N\ll X^{-1+\ee}{Q^3}$. As in \cite[Section 3.2]{Blomer-Li-2025}, combine the divisor main term with the dual divisor terms by allowing an index $m\in\Z$, with $m=0$ representing the main term. The resulting weight satisfies
\begin{align}	\label{eq:s1-weight-bound}
\Omega_{n_1,n_2}(m,q)\ll_A X^{\ee}\Big(1+\frac{|m|X}{q^2}\Big)^{-A}\Big(1+\frac{n_1^2n_2X}{q^3}\Big)^{-A}.
\end{align}
Consequently $|m|\ll X^\ee Q^2/X$ on the effective support.

After collecting the standard powers from the two Voronoi transformations, a dyadic block $S_1(N,H)$ is bounded, up to $X^\ee$, by an expression of the form
\begin{align}\label{eq:s1-dyadic-block}
\frac{X^{3/2}}{\delta LQ^{3/2}}\int_{|z|\ll\delta}\sum_{q=pt\asymp Q}\sum_{\substack{n_1\asymp H\\n_1\mid q}}\sum_{n_2\asymp N}\frac{|\lambda_\pi(n_2,n_1)|}{\sqrt{n_2}}\Big|\sum_mS_{h,m,n_1,\pm n_2}(q)\Omega_{n_1,n_2}(m,q)\Big|\,dz.
\end{align}
In the large-$H$ range this form is estimated pointwise.  For the small-$H$ range, we return to the expression preceding the triangle inequality in the $q=pt$-sum and keep the distinguished prime sum inside the absolute value.
 
We now consider the large-$H$ range. Applying Lemma \ref{lem:blomer-li-pointwise}, summing $m$ by \eqref{eq:s1-weight-bound}, and using \eqref{lem:hecke-moments}, one obtains (see \cite[(3.21)]{Blomer-Li-2025})
\begin{align*}
S_1(N,H)\ll X^\ee\Big(\frac XH+\frac{Q^2}{H}\Big).
\end{align*}
Thus, after fixing the constants in the dyadic partition, the range $H\ge Q_1/10$ satisfies
\begin{align*}
S_1(N,H)\ll X^\ee\Big(\frac X{Q_1}+\frac{Q^2}{Q_1}\Big).
\end{align*}
Under $Q\ge X^{1/2}$, the first term is dominated by $X^{1/2}Q/Q_1^{1/2}$, while the second is the middle term of $S_1$ in \cite{Blomer-Li-2025}.

We then consider the range $H\le Q_1/10$. 
Since $q=pt$, $p\asymp Q_1$, and $n_1\mid q$, condition $H\le Q_1/10$ forces $p\nmid n_1$ and therefore $n_1\mid t$. Together with $p\nmid h$, this verifies $(hn_1,p)=1$, the coprimality hypothesis needed in Lemma \ref{lem:blomer-li-correlation}.

At this point, the distinguished prime sum must be kept inside the absolute value. After fixing $t,n_1$ and the remaining smooth parameters, the relevant structure is
\begin{align*}
\sum_{n_2}\omega_{n_2}\sum_{p\asymp Q_1}\sum_mS_{h,m,n_1,\pm n_2}(pt)\Omega_{n_1,n_2}(m,pt).
\end{align*}
Apply Cauchy--Schwarz to the $n_2$-sum, use the Rankin--Selberg theory, and open the square. The resulting inner sum contains two primes $p_1,p_2\asymp Q_1$, two divisor-dual variables $m_1,m_2$, and a product of two $S$-sums. Poisson summation in $n_2$ modulo $[p_1,p_2]t$ gives
\begin{align*}
\frac{N}{[p_1,p_2]t}\sum_{k\in\Z}T(h,m_1,m_2,n_1,k,p_1,p_2,t)I(k),
\end{align*}
where $T$ is Lemma \ref{lem:blomer-li-correlation}, up to harmless unit rescalings, and $I(k)$ is the Fourier transform of the smooth $n_2$-weight.  Apply Lemma \ref{lem:blomer-li-correlation}. The zero Poisson frequency is present only when $p_1=p_2$; its diagonal $m_1=m_2$ is separated, and the remaining zero-frequency contribution contains $(m_1-m_2,p_1)$. The nonzero frequencies are bounded by Lemma \ref{lem:blomer-li-correlation}. Using the ranges of $m_i$, the elementary estimate
\begin{align*}
\sum_{|m_1|,|m_2|\ll M}(m_1-m_2,p)\ll M^2+Mp
\end{align*}
with the diagonal treated separately, one obtains (see also \cite[(3.27)--(3.29)]{Blomer-Li-2025})
\begin{align*}	
\begin{aligned}
S_1(N,H)\ll&X^\ee\Big(\frac{X^{1/2}Q}{Q_1^{1/2}}+\frac{Q^2}{Q_1}+\frac{X^{1/2}Q}{Q_2^{1/4}}\Big).
\end{aligned}
\end{align*}
Combining the small- and large-$H$ ranges proves (see \cite[(3.29)]{Blomer-Li-2025})
\begin{align}	\label{eq:s1-final-bound}
S_1\ll X^\ee\Big(\frac{X^{1/2}Q}{Q_1^{1/2}}+\frac{Q^2}{Q_1}+\frac{X^{1/2}Q}{Q_2^{1/4}}\Big).
\end{align}

\subsection{The endgame}

Fix the exponent $\ee>0$ in the theorem. Choose $0<\eta<\min(\ee/10,1/100)$ and take
\begin{align*}
Q_1=X^{2/11},\quad Q_2=X^{4/11},\quad Q=X^{6/11},\quad C_0=X^{1/2-\eta},\quad C=X^{1/2+\eta}.
\end{align*}
The major-arc hypotheses \eqref{eq:major-arc-range} follow directly from this choice. The remaining support and injectivity conditions follow from the identities
\begin{align*}
\frac{C_0^2}{X}=\frac{X}{C^2}=X^{-2\eta},\quad\frac{C_0Q_1}{X}\ll\frac{C_0Q_1^2}{X}=\frac{C_0Q_2}{X}=X^{-\f{3}{22}-\eta},\quad\frac{C_0}{Q}=X^{-\f{1}{22}-\eta}.
\end{align*}
After $\eta$ is fixed, choose the auxiliary constants in $\delta=X^{-1+\ee_0}$, the $z$-truncation, and the Poisson truncation sufficiently small, in particular with $\ee_0+\eta_0<\eta$ and $\ee_1<\ee_0/10$.
Indeed, under the above choice of parameters,
\begin{align*}
C_0^2\delta X^{\eta_0},\quad C_0Q_1^2\delta X^{\eta_0},\quad C_0Q_2\delta,\quad C_0Q_1\delta
\end{align*}
are all $o(1)$. Hence the first two branches have no nonzero circle frequency, the packed row map is injective, and the zero-frequency reindexing has no endpoint contribution. Moreover, $4C_0<Q$ for sufficiently large $X$, so the zero circle frequency vanishes, while $Q/X^{1/2}=X^{1/22}>1$, as required in the $S_1$ estimate. Finally, the effective major-arc range satisfies $|z|/\delta\ll X^{\ee_1-\ee_0}=o(1)$.

Under this choice, it follows from Proposition \ref{prop:s2-major-arcs} and \eqref{eq:s1-final-bound} that
\begin{align*}
S_1+S_2\ll X^{21/22+\eta+\ee/2}\ll X^{21/22+\ee}.
\end{align*}
This proves Theorem \ref{B-L}.

\section{Preliminaries on binary quadratic forms}

Let $\mathcal{Q}$ be a fixed positive definite integral binary quadratic form. We write
\begin{align*}
\mathcal{Q}(\mathbf{x})=\frac{1}{2}\mathbf{x}'A\mathbf{x},\quad D=\det A,
\end{align*}
where $A$ is a symmetric positive definite integral $2\times2$ matrix with even diagonal entries. Let
\begin{align*}
r(n,\mathcal{Q})=\#\big\{\mathbf{x}\in\mathbb{Z}^2:\mathcal{Q}(\mathbf{x})=n\big\}.
\end{align*}
For every $\ee>0$, we have the standard bound
\begin{align}\label{indi-r}
r(n,\mathcal{Q})\ll_{\mathcal{Q},\ee} n^\ee.
\end{align}

We first recall the Voronoi summation formula for $r(n,\mathcal{Q})$ in the form needed below. Given $D$ and $q$, we define 
\begin{align*}
\delta_0=\prod_{{\tt ord}_pD\leq{\tt ord}_pq}p^{{\tt ord}_pD},\quad\delta_1=\prod_{{\tt ord}_pD>{\tt ord}_pq}p^{{\tt ord}_pq}.
\end{align*}
Then $(D,q)=\delta_0\delta_1.$  We further write $D_0:=D/\delta_0$, and $q_1:=q/\delta_1$.
In addition, let
\begin{equation*}
G_{\mathcal{Q}}(q,a)=\sum_{\mathbf{x} \bmod q}e\Big(\f{a\mathcal{Q}(\mathbf{x})}{q}\Big)
\end{equation*}
be a Gaussian sum, where the condition of summation means that each component of ${\bf x}$ runs over a complete system of residues $\bmod q$.

The following is the binary specialization of the Voronoi formula in \cite[Theorem 1.4]{HL-exp-sums}.

\begin{lemma}\label{prop-vonoi-formula}
Let $(q,a)=1,$ and $\phi$ be a smooth function compactly supported on $(0,\infty)$. We have
\begin{align*}
\sum_{n=1}^{\infty}r(n,\mathcal{Q})e_q(an)\phi(n)={}&\frac{2\pi}{q^2\sqrt{D}}\Big(G_{\mathcal{Q}}(q,a)\widetilde{\phi}(1)+
\sum_{n=1}^{\infty}\widetilde{r}_{\mathcal{Q}^{*}}(n)e_{q}\big(-\overline{aD_0}n\big)F\Big(\f{n}{q^2}\Big)\Big),
\end{align*}
where $\widetilde{\phi}$ is the Mellin transform of $\phi$, $\widetilde{r}_{\mathcal{Q}^{*}}(n)$ is defined by
\begin{equation*}
\widetilde{r}_{\mathcal{Q}^{*}}(n):=\sum_{\substack{\mathcal{Q}^{*}({\bf x})=n}}\;\sum_{{\bf y} \bmod q}e_q(a\mathcal{Q}({\bf y}))e_{\delta_1}({\bf y}\cdot gVB{\bf x}),
\end{equation*}  
and
\begin{equation*}	
F\Big(\f{n}{q^2}\Big):=\int_0^\infty \phi(x)J_0\Big(\frac{4\pi\sqrt{nx}}{q\sqrt{D_0}}\Big)dx.
\end{equation*}  
Here $g$ is an integer such that $gq_1\equiv1(\bmod \delta_1)$, and $V,B,Q^*$, and $\widetilde r_{Q^*}$ are as defined in \cite[Theorem 1.4]{HL-exp-sums}.
\end{lemma}
From the definition of $\widetilde{r}_{\mathcal{Q}^{*}}(n)$ and the standard square-root bound for the inner complete quadratic sum, we shall use
\begin{align*}
\widetilde{r}_{\mathcal{Q}^{*}}(n)\ll_{\mathcal{Q},\ee}q^{1+\ee}r(n,\mathcal{Q}^{*})\ll_{\mathcal{Q},\ee}q^{1+\ee}n^\ee.
\end{align*}
On the factorable Jutila moduli we shall impose $(q,2D)=1$. In this case the preceding formula takes the simpler form
\begin{align}\label{Voronoi-r(n,Q)-symmetric}
\sum_{n=1}^{\infty}r(n,\mathcal{Q})e_q(an)\phi(n)={}&\frac{2\pi G_{\mathcal{Q}}(q,a)}{q^2\sqrt{D}}\Big(\widetilde{\phi}(1)+\sum_{n=1}^{\infty}r(n,\mathcal{Q}^{*})e_q(-\overline{aD}n)F\Big(\f{n}{q^2}\Big)\Big).
\end{align}
By virtue of the flexibility in choosing moduli afforded by Jutila’s circle method, we prefer to adopt the Voronoi type summation formula \eqref{Voronoi-r(n,Q)-symmetric} in our analysis of the contribution from $S_1$.
For odd $q$ coprime to $D$, the quadratic Gauss factor is
\begin{align}\label{eq:quadratic-gauss-evaluation}
G_\mathcal{Q}(q,a)=q\ee_q^2\Big(\frac Dq\Big)=q\Big(\frac{-D}{q}\Big),\qquad (a,q)=1,
\end{align}
where $\ee_q=1$ for $q\equiv1\bmod4$ and $\ee_q=i$ for $q\equiv3\bmod4$. The next observation is elementary, but it is the structural reason why the refined divisor-case argument can be transferred to the binary quadratic-form problem.

\begin{lemma}\label{lem:binary-gauss-reduction}
Assume $(q,2D)=1$. Then there is a number $\kappa_\mathcal{Q}(q)$, depending only on $q$ and the fixed form, such that $|\kappa_\mathcal{Q}(q)|=1$ and	$G_\mathcal{Q}(q,a)=\kappa_\mathcal{Q}(q)q$	for every reduced residue class $a\bmod q$. In particular, we have $q^{-2}G_\mathcal{Q}(q,a)=q^{-1}{\kappa_\mathcal{Q}(q)}$.
\end{lemma}

\begin{proof}
Diagonalize the nondegenerate binary form over each odd prime-power factor of $q$ and apply the one-dimensional quadratic Gauss sum evaluation. The product of the two quadratic characters of $a$ is $1$, whereas the determinant character remains. Chinese remaindering gives \eqref{eq:quadratic-gauss-evaluation}, and the assertion follows with $\kappa_\mathcal{Q}(q)=(-D/q)$.
\end{proof}

For the unrestricted Kloosterman-circle moduli we shall use the following complete-sum estimate. Let $c_0$ be the largest divisor of $c$ supported on the primes dividing $2D$, and write $c/c_0=c_2c_3$,
where $c_2$ is squarefull, $c_3$ is squarefree, and $(c_2,c_3)=1$. For $m_1\mid c$, up to harmless choices of signs, we define
\begin{align*}	
C_\mathcal{Q}(m_1,m_2;h,u;c)=\sideset{}{^{*}}\sum_{b\bmod c}e_c(-hb+\bar bu)G_\mathcal{Q}(c,b)S(-\bar b,\pm m_2;c/m_1).
\end{align*}

\begin{lemma}\label{lem:quadratic-complete-sum}
With the notation above, we have
\begin{align*}
C_\mathcal{Q}(m_1,m_2;h,u;c)\ll_{\mathcal{Q},\ee}\f{(c_0c_2)^{5/2+\ee}c_3^{2+\ee}(c_3,h)^{1/2}}{\sqrt{m_1}}.
\end{align*}
\end{lemma}

\begin{proof}
Chinese remaindering factors the complete sum into local factors, up to multiplication of the local parameters by units.	We first consider a prime $p\mid c_3$. Since $p\nmid2D$ and $p$ occurs to the first power in $c$, Lemma \ref{lem:binary-gauss-reduction} gives
 $G_{\mathcal{Q}}(p,b)=\kappa_{\mathcal{Q}}(p)p$ with $(b,p)=1$.
 
If $p\nmid m_1$, Lemma \ref{lem:blomer-li-pointwise} bounds the remaining complete character sum by $\ll p$. Thus the local factor is
$\ll p^2\le p^2(p,h)^{1/2}.$ If $p\mid m_1$, then the local Kloosterman modulus is equal to $1$. The remaining reduced residue-class sum is, up to unit changes of the parameters,
\begin{align*}
\sideset{}{^{*}}\sum_{b\bmod p}e_p(-hb+\bar bu)=S(-h,u;p).
\end{align*}
By the Weil bound,  we derive $|S(-h,u;p)|\ll p^{1/2}(p,h,u)^{1/2}\le p^{1/2}(p,h)^{1/2}$. Restoring the Gauss factor gives $\ll p^{3/2}(p,h)^{1/2}={p^2(p,h)^{1/2}}{{(m_1,p)}^{-1/2}}.$
Hence, uniformly for every $p\mid c_3$, the local contribution is bounded by ${p^2(p,h)^{1/2}}{{(m_1,p)}^{-1/2}}$.

We next consider an exact prime-power factor $v=p^\nu\Vert c_0c_2$ and put $a_v=(m_1,v)$. Since $\mathcal{Q}$ is fixed, the quadratic Gauss sum satisfies $|G_{\mathcal{Q}}(v,b)|\ll_{\mathcal{Q}} v$ with $(b,v)=1$. Moreover, the Weil bound gives	$|S(\bar b,m_2;v/a_v)|\ll \tau(v){v^{1/2}}{a_v}^{-1/2}$. Summing trivially over the reduced residue classes $b\bmod v$, we obtain $|C_{\mathcal{Q},v}|\ll v^{5/2+\ee}a_v^{-1/2}$.

Multiplying the local estimates and using
\begin{align*}
\prod_{p^\nu\Vert c_0c_2}(m_1,p^\nu)\prod_{p\mid c_3}(m_1,p)=m_1,
\end{align*}
we then conclude the desired bound. This completes the proof of the lemma.
\end{proof}

The nonzero Poisson frequencies will again be retained as a single family of matrix rows. The only new feature is the presence of the normalized quadratic Gauss multiplier.

\begin{lemma}\label{lem:weighted-operator-bound}
For $m_1\mid c$, put $q=c/m_1$ and $w_c(b)={G_{\mathcal{Q}}(c,b)}/{c}$,	and define
\begin{align*}
S^{\mathcal{Q}}_{h,d,m_1,m_2}(c)=\sideset{}{^{*}}\sum_{b\bmod c}w_c(b)e_c(hb+d\bar b)S(\bar b,m_2;q).
\end{align*}
Then, for arbitrary complex coefficients $\gamma_d$, we have
\begin{align*}
\sum_{m_2\bmod q}\Big|\sum_{d\bmod c}\gamma_d S^{\mathcal{Q}}_{h,d,m_1,m_2}(c)\Big|^2\ll_{\mathcal{Q}}\f{c^3}{m_1}\sum_{d\bmod c}|\gamma_d|^2.
\end{align*}
The same conclusion holds when $w_c(b)$ is replaced by $w_c(-b)$ or by $\overline{w_c(b)}$.
\end{lemma}

\begin{proof}
For $(b,c)=1$, put
\begin{align*}
B(b)=w_c(b)e_c(hb)\sum_{d\bmod c}\gamma_de_c(d\bar b).
\end{align*}
Opening the Kloosterman sum and applying orthogonality in $m_2$, we obtain
\begin{align*}
\sum_{m_2\bmod q}\Big|\sum_d\gamma_d S^{\mathcal{Q}}_{h,d,m_1,m_2}(c)\Big|^2=q\sideset{}{^{*}}\sum_{x\bmod q}\Big|\sideset{}{^{*}}\sum_{b\bmod c}B(b)e_q(\bar bx)\Big|^2.
\end{align*}
Enlarging the $x$-sum to all residue classes, Parseval modulo $q$ gives
\begin{align*}
\le q^2\sum_{v\bmod q}\Big|\sideset{}{^{*}}\sum_{\substack{b\bmod c\\\bar b\equiv v \bmod q}}B(b)\Big|^2.
\end{align*}
Each fibre contains at most $m_1$ elements. Hence Cauchy's inequality gives
\begin{align*}
\leq q^2m_1\sideset{}{^{*}}\sum_{b\bmod c}|B(b)|^2.
\end{align*}
Using
\begin{align*}
|B(b)|^2\le\|w_c\|_\infty^2\Big|\sum_{d\bmod c}\gamma_de_c(d\bar b)\Big|^2
\end{align*}
and enlarging the reduced residue-class sum to all classes, Parseval modulo $c$ yields
\begin{align*}
q^2m_1\sideset{}{^{*}}\sum_{b\bmod c}|B(b)|^2\le q^2m_1c\,\|w_c\|_\infty^2\sum_{d\bmod c}|\gamma_d|^2.
\end{align*}
For reduced $b$, the standard quadratic Gauss-sum estimate gives $|G_{\mathcal{Q}}(c,b)|\ll_{\mathcal{Q}}c$, and hence $\|w_c\|_\infty\ll_{\mathcal{Q}}1$. Since $q=c/m_1$, we have	$q^2m_1c={c^3}/{m_1}$. Therefore, we obtain
\begin{align*}
\sum_{m_2\bmod q}\Big|\sum_d\gamma_dS^{\mathcal{Q}}_{h,d,m_1,m_2}(c)\Big|^2
\ll_{\mathcal{Q}}\f{c^3}{m_1}\sum_{d\bmod c}|\gamma_d|^2.
\end{align*}
Only the supremum norm of the multiplier has been used, so the variants with $w_c(-b)$ and $\overline{w_c(b)}$ follow in the same way.
\end{proof}

\section{Proof of Theorem \ref{T1}}

We now turn to the proof of Theorem \ref{T1}. The argument follows the proof of Theorem \ref{B-L}, and we record  the points where the binary quadratic-form setting requires a modification. There are three such points. First, on the unrestricted Farey moduli, the normalized quadratic Gauss sum $w_c(b)$ is retained inside the complete character sum. Lemma \ref{lem:weighted-operator-bound} therefore replaces Lemma \ref{lem:uniform-operator-bound} in the treatment of the packed nonzero Poisson frequencies. Second, the zero Poisson frequency contains finitely many additional local factors at the primes dividing $2D$; these factors will be kept explicitly and will cancel against the corresponding local density in the Jutila normalization. Third, on the factorable Jutila moduli, which are chosen coprime to $2D$, Lemma \ref{lem:binary-gauss-reduction} makes the quadratic Gauss factor independent of the reduced residue class. Consequently, after the quadratic-form Voronoi transformation, the same complete character sums as in Section \ref{sec-S1-BL} reappear, with the same effective dual length.

Let $\phi$ be the smooth weight occurring in Theorem \ref{T1}. Choose
\begin{align*}
G\in C_c^\infty((1/2,4)),\quad G(x)=1\quad\text{for }1\le x\le3,
\end{align*}
and put	$g(x)=G(x/X)$. 
Since $\phi$ is supported on $[X,2X]$ and $0<h\le X$, we have $g(n+h)=1$ throughout the support of the original shifted convolution sum. Hence additive orthogonality gives
\begin{align*}
S_h^{\mathcal{Q}}(X)=\sum_{m,n}\lambda_\pi(1,m)g(m)r(n,\mathcal{Q})\phi(n)\int_0^1e((n+h-m)\alpha)\,d\alpha.
\end{align*}
For the Jutila function, we take the factorable weight
\begin{align}\label{eq:quadratic-jutila-weight}
w_\mathcal{Q}(q)=\sum_{\substack{Q_1/2\le p\le Q_1\\ p\ \mathrm{prime},\ p\nmid 2Dh}} \sum_{\substack{pt=q\\(t,2D)=1}}\rho(t/Q_2),
\end{align}
where $Q=Q_1Q_2$. Because $D$ is fixed, deleting the primes dividing $2D$ and restricting $t$ to $(t,2D)=1$ changes only the local density. The same elementary counting as in $L$ gives
\begin{align*}	
L_\mathcal{Q}=\sum_q\varphi(q)w_\mathcal{Q}(q)=Q^{2+o(1)},\quad \|w_\mathcal{Q}\|_\infty\ll X^\ee.
\end{align*}
Thus Lemma \ref{Jutila circle method} applies with the same powers of $X$ as in the divisor case.

Let $\chi$ denote the corresponding Jutila function. Inserting $1=\chi+(1-\chi)$,
we write
\begin{align*}
S_h^{\mathcal{Q}}(X)=S_1^{\mathcal{Q}}+S_2^{\mathcal{Q}},
\end{align*}
where
\begin{align*}
S_1^{\mathcal{Q}}={}&\sum_{m,n}\lambda_\pi(1,m)g(m)r(n,\mathcal{Q})\phi(n)\int_0^1\chi(\alpha)e((n+h-m)\alpha)\,d\alpha,
\end{align*}
and
\begin{align*}
S_2^{\mathcal{Q}}={}&\sum_{m,n}\lambda_\pi(1,m)g(m)r(n,\mathcal{Q})\phi(n)\int_0^1(1-\chi(\alpha))e((n+h-m)\alpha)\,d\alpha.
\end{align*}

The treatment of $S_2^{\mathcal{Q}}$ follows Section \ref{sec-S2-BL}. The packed nonzero Poisson frequencies are handled by Lemma \ref{lem:weighted-operator-bound}, while Lemma \ref{lem:quadratic-complete-sum} is used for the remaining pointwise contribution. The only additional issue is the zero Poisson frequency, where the finitely many local factors at the primes dividing $2D$ must be retained until the arithmetic mean is compared with the normalization $L_{\mathcal{Q}}$. The treatment of $S_1^{\mathcal{Q}}$ follows Section \ref{sec-S1-BL}. Since every Jutila modulus is coprime to $2D$, Lemma \ref{lem:binary-gauss-reduction} reduces the quadratic Gauss factor to a scalar before Cauchy--Schwarz is applied. We therefore record below only the modifications required in the two parts of the argument.

\subsection{Estimation of $S_2$}\label{sec-S2}

We now estimate $S_2^{\mathcal{Q}}$. The argument follows Section \ref{sec-S2-BL}, and we record only the modifications arising from the quadratic-form Voronoi formula. Applying Lemma \ref{Kloosterman circle method}, we write
\begin{align}\label{S2-def}
\begin{aligned}
S_2^\mathcal{Q}=& \sum_n \sum_m \lambda_\pi(1,m) g(m) r(n,\mathcal{Q}) \phi(n)\\
&\times \sum_{c\le C}\; \sideset{}{^*}\sum_{{b\bmod c}}\int_{-\frac{1}{c(c+c')}}^{\frac{1}{c(c+c'')}}
\Big(1-\chi\Big(\frac{b}{c}+z\Big)\Big) e\Big((n+h-m)\Big(\frac{b}{c}+z\Big)\Big) dz.
\end{aligned}
\end{align}
Next, we choose parameters satisfying $C_0 < C$ with $\log C_0 \asymp \log C \asymp \log X$. To derive an upper bound for $S_2^{\mathcal Q}$, we decompose $S_2^{\mathcal Q}$ further. We split $S_2^\mathcal{Q}= S_{2,0}^\mathcal{Q} + S_{2,1}^\mathcal{Q}$,
where $S_{2,0}^\mathcal{Q}$ stands for the contribution from major arcs with $c\le C_0$, while $S_{2,1}^\mathcal{Q}$ corresponds to the minor arcs with $C_0< c\leq C$.

For the minor arc part, the same Cauchy--Schwarz argument as in Section \ref{sec-S2-BL}, together with Lemma \ref{Jutila circle method} and Rankin--Selberg theory, gives
\begin{align*}
S_{2,1}^\mathcal{Q}&\ll\Big(\frac{Q^2}{L_\mathcal{Q}^2\delta}\Big)^{\frac12} X^{\frac12+\ee}\max_{C_0<c\le C}\;\max_{\substack{b\bmod c\\(b,c)=1}}\;\max_{ |z|\le \frac{1}{cC}}\Big|\sum_n r(n,\mathcal{Q})\phi(n) e\Big(n\Big(\frac{b}{c}+z\Big)\Big)\Big|.
\end{align*}
By Lemma \ref{prop-vonoi-formula}, the summation over $n$ becomes
\begin{align*}
\sum_n r(n,\mathcal{Q})\phi(n)e(nz)e\Big(\frac{nb}{c}\Big)= &\frac{2\pi}{c^2\sqrt{D}}\Big(G_\mathcal{Q}(c,b)\widetilde{\phi_z}(1)+
\sum_{n=1}^\infty {\widetilde{r}_{^\mathcal{Q*}}(n)}e\Big(-\frac{\overline{bD_0}{n}}{c}\Big)F_z\Big(\frac{n}{c^2}\Big)\Big),
\end{align*}
where 
\begin{equation*}	
F_z\Big(\f{n}{c^2}\Big)=\int_0^\infty \phi(x)e(xz)J_{0}\Big({\frac{4\pi\sqrt{nx}}{c\sqrt{D_0}}}\Big)dx.
\end{equation*}  
As in \eqref{eq:divisor-minus-transform}, an application of the method of stationary phase yields the bound
\begin{align*}
F_z\Big(\f{n}{c^2}\Big)\ll \f{X^{1+\ee}}{1+|z|X}\Big(1+\f{nX}{c^2(1+|z|X)^2}\Big)^{-A}.
\end{align*}
 Integration by parts, we derive
\begin{equation*}	
\widetilde{\phi_z}(1)\ll \f{X}{1+|z|X}.
\end{equation*} 
Upon invoking the square-root cancellation for the Gauss sum $G_\mathcal{Q}(c,b)$ together with the individual upper bound of the arithmetic function $r(n,\mathcal{Q}^*)$ in \eqref{indi-r}, we deduce 
\begin{align}\label{S_21-bound}
\begin{aligned}
S_{2,1}^\mathcal{Q}\ll\frac{QX^{\frac12+\ee}}{L_\mathcal{Q}\delta^{\frac12}}\max_{C_0<c\le C}\max_{ |z|\le \frac{1}{cC}}\Big( \frac{X}{c}+c(1+|z|X) \Big)\ll\frac{X^{\frac12+\ee}}{Q\delta^{\frac12}}\Big( \frac{X}{C_0}+C\Big).
\end{aligned}
\end{align}

We now revert to \eqref{S2-def}, and proceed to examine the major arc contribution $S_{2,0}^\mathcal{Q}$ under the constraint $c\le C_0$. Suppose $c\le C_0$ and use \eqref{eq:refined-farey}. The nonzero quadratic-form dual sum has length
\begin{align*}
n\ll X^\ee\f{c^2(1+|z|X)^2}{X} \ll X^\ee\Big(\frac{C_0^2}{X}+\frac{X}{C^2}\Big).
\end{align*}
Under the hypotheses $C_0\le X^{1/2-\eta}, C\ge X^{1/2+\eta},$ this part is negligible. The zero term in the quadratic-form transformation is
\begin{align*}
\frac{2\pi }{c\sqrt D}\,w_c(b)\int_0^\infty\phi(x)e(zx)\,dx, \text{ with }  w_c(b)=\frac{G_\mathcal{Q}(c,b)}c.
\end{align*}
Repeated integration by parts restricts $|z|\ll X^{-1+\ee_1}$. Thus the same normalization as in \eqref{eq:s2-major-arc-master} applies, with the factor $w_c(b)$ retained inside the complete sum. Reflection $b\mapsto-b$ replaces $w_c(b)$ by $w_c(-b)$; Lemma \ref{lem:weighted-operator-bound} applies to either multiplier. The smooth quadratic-form zero term and the $\mathrm{GL}(3)$ Voronoi weight are independent of the variables subsequently introduced from $\chi$.

Denote the contribution of $1$ in $1-\chi$ by $S_{2,0,0}^\mathcal{Q}$ and retain it without estimating it separately. Below it will be combined with the zero Poisson frequency from $\chi$. These two multipliers act on the same normalized character sum and the same smooth transforms.

For the $\chi$-part, insert \eqref{chi-repr}, with the weight $w_\mathcal{Q}$ and normalizing constant $L_\mathcal{Q}$. 
The entire multiplier is $\psi(z/\delta)w_\mathcal{Q}(c)/(\delta L_\mathcal{Q})$. We also assume $4C_0<Q$, as holds for the parameters chosen below. Since $w_\mathcal{Q}$ is supported on $[Q/4,Q]$, this multiplier vanishes for $c\le C_0$. Therefore, we have
\begin{align}	\label{eq:quadratic-zero-circle-frequency}
S_{2,0,1}^{\mathcal{Q},l=0}=0.
\end{align}

Next, we treat the nonzero circle frequencies. Put $R=\rad(2D)=\textstyle\prod_{p\mid 2D} p$. We impose the support and injectivity conditions
\begin{align*}
C_0Q_2\delta=o(1),\quad C_0Q_1\delta=o(1),\quad C_0^2\delta X^{\eta_0}=o(1),\quad C_0Q_1^2\delta X^{\eta_0}=o(1),
\end{align*}
which are verified at the end of the proof. The second condition also ensures that the support of every nonzero circle frequency stays away from the lower endpoint of the positive $d$-sum. The branch decomposition  remains valid for $w_\mathcal{Q}$; the additional coprimality restrictions only remove terms. In the first two branches, the support condition is $0<|\ell|\ll cY\delta$, with $Y\ll Q_2/(g_1r)$, so $C_0Q_2\delta=o(1)$ eliminates both branches. In the third branch we have
\begin{align*}
c=g_1g_2,\quad p\asymp Q_1,\quad p\nmid chR,\quad (p,r)=1,\quad d\asymp Y\ll\frac{Q_2}{g_1r}.
\end{align*}
After dividing the original circle frequency by $g_1$, we get
\begin{align*} 
(l,g_2)=1,\quad d\equiv\bar p\bar bl\bmod{g_2},\quad 0<|l|\ll cYp\delta.
\end{align*}
Here $|z|/\delta=o(1)$, so the number of admissible nonzero $l$ is $O(cYQ_1\delta)$, without an additional constant term. The condition on the second factor in \eqref{eq:quadratic-jutila-weight} is exactly $(dg_1r,R)=1$. Therefore $(g_1r,R)=1$, and we expand the remaining restriction as
\begin{align*}
\1_{(d,R)=1}=\sum_{\substack{s\mid R\\s\mid d}}\mu(s).
\end{align*}
There are only $O_\mathcal{Q}(1)$ choices of $s$. If $(s,g_2)>1$, the congruence for $d$ has no solution. For $(s,g_2)=1$, we also have $(s,g_1)=1$ since $(g_1,R)=1$. Thus $(s,c)=1$, and the inverses below may be taken modulo $c$. Write $d=sv$. Then	$v\asymp Y_s:=Y/s$ with $v\equiv\overline{sp}\bar b l\bmod{g_2}$.
With $v_0$ a fixed dyadic cutoff, Poisson summation in $v$ gives the coefficient (as in \eqref{eq:poisson-weight})
\begin{align}	\label{eq:quadratic-poisson-coefficient}
\alpha_{p,l,\nu}^{(s)}=\frac1{g_2}\int_{\R}v_0(s\xi/Y)\rho(s\xi g_1r/Q_2)\psi\Big(\delta^{-1}\Big(\frac{l}{csp\xi}+z\Big)\Big)e\Big(-\frac{\nu\xi}{g_2}\Big)\,d\xi.
\end{align}
Every derivative of order $j$ of its smooth integrand is $O_{j,\mathcal{Q}}(Y_s^{-j})$. Hence, with $H_s=g_2/Y_s=sg_2/Y$, we obtain
\begin{align}	\label{eq:quadratic-poisson-energy}
|\alpha_{p,l,\nu}^{(s)}|\ll\frac{Y}{sg_2}\Big(1+\frac{|\nu|Y}{sg_2}\Big)^{-A}, \quad\sum_{\nu\ne0}|\alpha_{p,l,\nu}^{(s)}|^2\ll\frac{Y}{sg_2}.
\end{align}
The frequencies $|\nu|>H_sX^{\eta_0}$ are negligible after choosing $A$ sufficiently large. If this upper limit is less than $1$, there is no effective nonzero Poisson frequency.

Next, we turn to the treatment of the nonzero Poisson frequencies. Fix $s\mid R$ with $(s,g_2)=1$, and for $0<|k|\ll sg_2cQ_1\delta X^{\eta_0}$, set
\begin{align*}
\beta_{p,k}^{(s)}=\sum_{\substack{l\nu=k\\l\ne0,\ \nu\ne0}}\alpha_{p,l,\nu}^{(s)}.
\end{align*}
The exact Poisson phase contributed by the congruence is $e_{g_2}(\nu l\overline{sp}\bar b)$. After the same change $b\mapsto-b$ as in \eqref{eq:exact-poisson-phase}--\eqref{eq:packed-coefficients}, the row parameter of the weighted matrix in Lemma \ref{lem:weighted-operator-bound} is $d_0-g_1k\overline{sp}\pmod c$ with $d_0=-u.$
All parameters $p,l,\nu$ thus occur in this row and in its coefficient; the normalized $\mathrm{GL}(3)$ weight is common to the family.

For fixed $s$, this row map is injective. In fact, we write $K_s\ll_\mathcal{Q} cg_2Q_1\delta X^{\eta_0}$ for the upper bound on $|k|$. The two injectivity conditions give $K_s<Q_1/2$ and $2Q_1K_s<g_2$ for large $X$. Equality of two rows, followed by multiplication by $sP_1p_2$ modulo $g_2$, implies $k_1p_2-k_2p_1\equiv0\bmod{g_2}$.
Its absolute value is less than $g_2$, so it vanishes as an integer. If $p_1\ne p_2$, primality would imply $p_1\mid k_1$, contrary to $0<|k_1|<Q_1/2\le p_1$. Hence $p_1=p_2$ and $k_1=k_2$.

Cauchy--Schwarz in $l\nu=k$, the divisor bound, and \eqref{eq:quadratic-poisson-energy} give
\begin{align}	\label{eq:quadratic-packed-energy}
\sum_{p,k}|\beta_{p,k}^{(s)}|^2\ll X^\ee Q_1(cYQ_1\delta)\frac{Y}{sg_2}=X^\ee\frac{Q_1^2g_1Y^2\delta}{s}.
\end{align}
There is no linear summation over $p$ or $l$ before this quadratic estimate. Let $T_s(m_2)$ be the corresponding sum of weighted matrix rows. By injectivity and Lemma \ref{lem:weighted-operator-bound}, we derive
\begin{align*}
\sum_{m_2\bmod c/m_1}|T_s(m_2)|^2\ll \frac{c^3}{m_1}\sum_{p,k}|\beta_{p,k}^{(s)}|^2.
\end{align*}
An effective block $m_2\asymp N$ satisfies $N/(c/m_1)\ll X^\ee c^2/(m_1X)=o(1)$. Thus its interval contains at most one representative of each residue class modulo $c/m_1$. Let $U(m_2)$ denote the common normalized Voronoi cutoff. Then Cauchy–Schwarz gives
\begin{align}\label{eq:quadratic-packed-cauchy}
\sum_{m_2\asymp N}\frac{|\lambda_\pi(m_2,m_1)|}{\sqrt m_2}|U(m_2)||T_s(m_2)|\ll X^\ee c^{3/2}\|\beta^{(s)}\|_2.
\end{align}

The zero quadratic-form transform and the $\mathrm{GL}(3)$ transform, the $z$-integration, and the insertion of $\chi$ supply the common factor
\begin{align*}
\frac{X^{1/2+\ee}}{\delta Q^2c^{3/2}}.
\end{align*}
The Poisson factor $g_2^{-1}$ has already been included in \eqref{eq:quadratic-poisson-coefficient}. Consequently \eqref{eq:quadratic-packed-energy}--\eqref{eq:quadratic-packed-cauchy} bound a fixed block by
\begin{align*}
X^\ee\frac{X^{1/2}Q_1Y\sqrt{g_1}}{\delta^{1/2}Q^2\sqrt s}\ll X^\ee\frac{X}{Q\sqrt{sg_1}\,r},
\end{align*}
where $Y\ll Q_2/(g_1r)$ and $\delta^{-1/2}\le X^{1/2}$. Summing $m_1\mid c$, $ g_1\mid c$, $r\ll Q_2$, the dyadic blocks, and $u$ costs only $X^\ee$. Finally, summing $c\le C_0$ proves (as in \eqref{eq:nonzero-poisson-total})
\begin{align}\label{eq:quadratic-nonzero-poisson-total}
S_{2,0,1}^{\mathcal{Q},l\ne0,\nu\ne0}\ll X^\ee\frac{XC_0}{Q}.
\end{align}

We next turn to the zero Poisson frequency and the analysis of its mean.  Let $\mathcal P_R$ be the primes $p\in[Q_1/2,Q_1]$ with $p\nmid Rh$. For fixed $c$, we put
\begin{align*}
B=(c,R),\quad \vartheta(M)=\frac{\varphi(M)}M,\quad I_\rho=\int_0^\infty t\rho(t/Q_2)\,dt.
\end{align*}
We sum these coefficients with their M\"obius signs over $s,g_1,r,l$ and the dyadic partition before taking absolute values. Denote the resulting scalar multiplier of $\chi$ by $\chi_{\mathcal{Q},0}(c,z)$.

The restrictions $(g_1r,R)=1$ and $(s,g_2)=1$ imply that the surviving $s$ are precisely the divisors of $R/B$. Changing $s\xi$ to $d$ in the zero-frequency integral therefore gives the factor
\begin{align*}
\sum_{s\mid R/B}\frac{\mu(s)}s=\vartheta(R/B).
\end{align*}
Next, let $t=g_1rd$ and $j=g_1r l$. The measure and progression factor become $dt/(cr)$. For fixed $r\mid j$, the conditions
\begin{align*}
g_1\mid c,\quad g_1\mid j/r,\quad \Big(\frac{j}{g_1r},\frac{c}{g_1}\Big)=1
\end{align*}
uniquely determine $g_1=(j/r,c)$. Since $(r,R)=1$, the additional condition $(g_1,R)=1$ is equivalent to $(j,B)=1$. It follows that, with
\begin{align*}
a_{p,R,B}(j)=\1_{(j,B)=1}\sum_{\substack{r\mid j\\(r,pR)=1}}\frac{\mu(r)}r,\quad j\ge1,
\end{align*}
the exact scalar is (as in \eqref{eq:zero-poisson-exact})
\begin{align*}
\chi_{\mathcal{Q},0}(c,z)={}&\frac{\vartheta(R/B)}{\delta L_\mathcal{Q}c}\sum_{\substack{p\in\mathcal P_R\\p\nmid c}}\int_0^\infty\rho(t/Q_2)\sum_{j\ne0}a_{p,R,B}(|j|)\psi\Big(\delta^{-1}\Big(\frac{j}{cpt}+z\Big)\Big)\,dt.
\end{align*}
There is no endpoint error in this recombination. Indeed, the support gives $0<|j|\ll cQ\delta$ and $g_1r\le|j|$, whence
\begin{align*}
	d=\frac{t}{g_1r}\gg\frac1{cQ_1\delta}\longrightarrow\infty.
\end{align*}
The original positive-integer cutoff and the cutoffs after division by any fixed $s\mid R$ are consequently identically one on the relevant support. The exact dyadic partition may therefore be summed to one. This argument also shows that every divisor $r$ appearing above belongs to the original range of summation.

For $T\ge1$, opening the divisor sum and counting integers coprime to $B$ give
\begin{align*}
\sum_{1\le j\le T}a_{p,R,B}(j)&=\sum_{\substack{r\le T\\(r,pR)=1}}\frac{\mu(r)}r\sum_{\substack{m\le T/r\\(m,B)=1}}1=T\vartheta(B)\sum_{\substack{r\ge1\\(r,pR)=1}}\frac{\mu(r)}{r^2}+O_R(\log(2T))\\
&=\frac{T\vartheta(B)}{\zeta(2)\prod_{l\mid pR}(1-l^{-2})}+O_R(\log(2T)).
\end{align*}
Here the tail of the absolutely convergent $r^{-2}$ sum contributes $O(1)$, and the errors from counting $m$ sum to $O_R(\log(2T))$. Partial summation separately on the positive and negative integers now shows that, uniformly for $Z>0$ and $|v|\le1/4$,
\begin{align*}
\sum_{j\ne0}a_{p,R,B}(|j|)\psi(j/Z+v)=\frac{Z\vartheta(B)}{\zeta(2)\prod_{l\mid pR}(1-l^{-2})}+O(\log(2+Z)).
\end{align*}
For bounded $Z$, the same formula follows directly; the main integrals from the two half-lines combine to $Z\int\psi=Z$. Since $R$ is squarefree, $\vartheta(R/B)\vartheta(B)=\vartheta(R)$. Define
\begin{align*}
\kappa_R=\frac1{\zeta(2)}\prod_{l\mid R}(1+l^{-1})^{-1}.
\end{align*}
Applying the weighted mean with $Z=cpt\delta$ and $v=z/\delta$ yields (as in Lemma \ref{eq:zero-poisson-main-term})
\begin{align*}
\chi_{\mathcal{Q},0}(c,z)=\frac{I_\rho}{L_\mathcal{Q}}\sum_{\substack{p\in\mathcal P_R\\p\nmid c}}\frac{p\kappa_R}{1-p^{-2}}+O\Big(\frac{X^\ee}{cQ\delta}\Big).
\end{align*}
In particular, all local factors depending on $(c,R)$ have disappeared.

To compare this mean with the Jutila normalization, use the exact identity \eqref{phi(pt)}. Restricting $t$ to $(t,R)=1$ forces $(r,R)=1$, and the same counting and partial summation give (as in Lemma \ref{eq:jutila-normalization-precise}) 
\begin{align*}
L_\mathcal{Q}=I_\rho\sum_{p\in\mathcal P_R}\frac{(p-1)\kappa_R}{1-p^{-2}}+O\Big(Q_2\log(2Q_2)\sum_{p\in\mathcal P_R}p\Big).
\end{align*}
As in \eqref{eq:zero-poisson-cancellation-detailed}, we conclude that
\begin{align}	\label{eq:quadratic-zero-poisson-cancellation}
|1-\chi_{\mathcal{Q},0}(c,z)|\ll X^\ee\Big(\frac1{cQ\delta}+\frac1{Q_1}+\frac1{Q_2}\Big).
\end{align}
We now combine this scalar with the original multiplier $1$ in $1-\chi$. The normalized Gauss factor $w_c(b)$ and both smooth transforms are common to these two terms. Thus the pointwise estimate  applies with the scalar $1-\chi_{\mathcal{Q},0}(c,z)$, even when $(c,2D)>1$. Their combined contribution is
\begin{align}\label{eq:quadratic-zero-poisson-combined}
S_{2,0,0}^\mathcal{Q}-S_{2,0,1}^{\mathcal{Q},l\ne0,\nu=0}\ll X^\ee\biggl(\frac{C_0}{Q\delta}+\frac{C_0^2}{Q_1}+\frac{C_0^2}{Q_2}\biggr).
\end{align}
Here $S_{2,0,1}^{\mathcal{Q},l\ne0,\nu=0}$ denotes the contribution with the positive multiplier $\chi_{\mathcal{Q},0}$, so the minus sign is the one in $1-\chi$. Since $\delta^{-1}\le X$, equations \eqref{eq:quadratic-zero-circle-frequency}, \eqref{eq:quadratic-nonzero-poisson-total}, and \eqref{eq:quadratic-zero-poisson-combined} imply
\begin{align}\label{eq:quadratic-major-arc-bound}
S_{2,0}^\mathcal{Q}\ll X^\ee\biggl(\frac{XC_0}{Q}+\frac{C_0^2}{Q_1}+\frac{C_0^2}{Q_2}\biggr).
\end{align}
Combining this with \eqref{S_21-bound}, we obtain
\begin{align}\label{eq:quadratic-s2-final-bound}
S_2^\mathcal{Q}\ll X^\ee\Big(\frac{X^2}{QC_0}+\frac{XC}{Q}+\frac{XC_0}{Q}+\frac{C_0^2}{Q_1}+\frac{C_0^2}{Q_2}\Big).
\end{align}
The argument uses Lemma \ref{lem:quadratic-complete-sum} for the remaining scalar contribution and Lemma \ref{lem:weighted-operator-bound} for the nonzero Poisson frequencies. It imposes no coprimality condition between the Farey modulus $c$ and $2D$.

\subsection{Estimation of  $S_1$}\label{sec-S1}

We now estimate $S_1^\mathcal{Q}$. Here every modulus in \eqref{eq:quadratic-jutila-weight} is coprime to $2D$, so the exact reduction of Lemma \ref{lem:binary-gauss-reduction} is available. From Lemma \ref{Jutila circle method}, we obtain
\begin{align*}	
S_1^\mathcal{Q}={}&\frac1{\delta L_\mathcal{Q}}\sum_q w_\mathcal{Q}(q)\sideset{}{^{*}}\sum_{a\bmod q}\int_{\R}\psi\Big(\f z\delta\Big)e_q(ah)e(zh)\sum_n r(n,Q)\phi(n)e_q(an)e(zn)\\
&\times\sum_m\lambda_\pi(1,m)g(m)e_q(-am)e(-zm)\,dz.
\end{align*}
Apply \eqref{Voronoi-r(n,Q)-symmetric} to the $n$-sum and Lemma \ref{voronoi3} to the $m$-sum. By Lemma \ref{lem:binary-gauss-reduction}, the factor $q^{-2}G_Q(q,a)$ becomes $\kappa_Q(q)/q$ and is independent of $a$. Thus the sum over $a$ contains, with the factor $\kappa_Q(q)$ retained in the coefficient,
\begin{align}	\label{eq:quadratic-character-sum-reduction}
\sideset{}{^{*}}\sum_{a\bmod q}e_q(ha+\bar a d)S(-\bar a,\pm m_2;q/m_1)=S_{h,d,m_1,\mp m_2}(q),
\end{align}
where $d=0$ for the main term of \eqref{Voronoi-r(n,Q)-symmetric} and $d=-\overline{D}n$ for a nonzero quadratic-form dual variable $n$; here $D\overline{D}\equiv1\pmod q$.  The change of sign on the right of \eqref{eq:quadratic-character-sum-reduction} follows from $S(-\bar a,\pm m_2;q/m_1)=S(\bar a,\mp m_2;q/m_1)$. Below we denote the resulting sign by $\sigma\in\{1,-1\}$ and treat the two signs separately.

It is useful to combine these two cases. Put
\begin{align*}
A_\mathcal{Q}(0;q,z)=\widetilde\phi_z(1),\quad A_\mathcal{Q}(n;q,z)=r(n,\mathcal{Q^*})F_{\phi,z}(n/q^2),\quad n\ge1.
\end{align*}
By \eqref{indi-r}, for every $A>0$, we have
\begin{align}	\label{eq:quadratic-dual-weight-bound}
A_\mathcal{Q}(n;q,z)\ll X^{1+\ee}\Big(1+\frac{nX}{q^2}\Big)^{-A},
\end{align}
with the same bound for $n=0$. Hence the effective range is	$0\le n\ll X^{-1+\ee}{Q^2}$.
This is exactly the length of the divisor-side dual variable in Section \ref{sec-S1-BL}.

Place the $\mathrm{GL}(3)$ divisor $m_1\asymp H$ and the remaining dual variable $m_2\asymp N$. The $\mathrm{GL}(3)$ transform gives
\begin{align}\label{eq:quadratic-gl3-dual-range}
H^2N\ll X^{-1+\ee}{Q^3}.
\end{align}
After removing the common factors from the two transformations, a dyadic block is bounded by the analogue of \eqref{eq:s1-dyadic-block}, namely
\begin{align}	\label{eq:quadratic-s1-dyadic-block}
\frac{X^{3/2+\ee}}{\delta L_\mathcal{Q}Q^{3/2}}\int_{|z|\ll\delta}\sum_{q=pt\asymp Q}\sum_{\substack{m_1\asymp H\\m_1\mid q}}\sum_{m_2\asymp N}\frac{|\lambda_\pi(m_2,m_1)|}{\sqrt{m_2}}\Big|\sum_{n\ge0}S_{h,-\overline{D}n,m_1,\sigma m_2}(q)\Omega_{m_1,m_2}^\mathcal{Q}(n,q)\Big|\,dz,
\end{align}
where, after normalizing by the factor $X$ in \eqref{eq:quadratic-dual-weight-bound},
\begin{align*}
\Omega_{m_1,m_2}^\mathcal{Q}(n,q)\ll X^\ee\Big(1+\frac{nX}{q^2}\Big)^{-A}\Big(1+\frac{m_1^2m_2X}{q^3}\Big)^{-A}.
\end{align*}
The $X$-power and every modulus power in \eqref{eq:quadratic-s1-dyadic-block} agree with the divisor block \eqref{eq:s1-dyadic-block}; this equality of normalizations is the main benefit of the binary Gauss reduction.

We keep the arithmetic coefficient distinct from the smooth transform. Specifically, set $b_\mathcal{Q}(0)=1$ and $b_\mathcal{Q}(n)=r(n,\mathcal{Q^*})$ for $n\ge1$, and write
\begin{align*}
\Omega_{m_1,m_2}^\mathcal{Q}(n,q)=\kappa_\mathcal{Q}(q)b_\mathcal{Q}(n)\Xi_{m_1,m_2}(n,q).
\end{align*}
The scalar $\kappa_\mathcal{Q}(q)$ has modulus one, and $b_\mathcal{Q}(n)\ll (1+n)^\ee$. For fixed $n,q,z,m_1$, the factor $\Xi$ is smooth in $m_2$. On a block $m_2\asymp N$, the differentiated Voronoi bounds give
\begin{align*}
N^j\Big|\frac{\partial^j}{\partial m_2^j}\Xi_{m_1,m_2}(n,q)\Big|\ll X^\ee\Big(1+\frac{nX}{q^2}\Big)^{-A}\Big(1+\frac{m_1^2NX}{q^3}\Big)^{-A}.
\end{align*}
As usual, the auxiliary exponent in $\delta=X^{-1+\ee_0}$ is chosen sufficiently small for the prescribed $\ee$ and the fixed number of integrations by parts. These estimates apply also to $n=0$. No derivative is taken of $b_\mathcal{Q}(n)$ or $\kappa_\mathcal{Q}(q)$.

When $H\ge Q_1/10$, apply Lemma \ref{lem:blomer-li-pointwise} to the complete sum in \eqref{eq:quadratic-s1-dyadic-block}. The calculation is identical in powers to the large-$H$ divisor calculation and gives
\begin{align*}	
S_1^\mathcal{Q}(N,H)\ll X^\ee\Big(\frac XH+\frac{Q^2}{H}\Big)\ll X^\ee\Big(\frac X{Q_1}+\frac{Q^2}{Q_1}\Big).
\end{align*}

Assume $H\le Q_1/10$. Since $q=pt$, $p\asymp Q_1$, and $m_1\mid q$, we have $p\nmid m_1$ and therefore $m_1\mid t$. Moreover $p\nmid hD$ by construction, so $(hm_1,p)=1$. We retain the prime sum inside the absolute value, exactly as in the divisor argument. Cauchy--Schwarz in $m_2$ followed by Rankin--Selberg opens a square containing two primes $p_1,p_2$ and two quadratic-form dual variables $n_1,n_2$.
Poisson summation in $m_2$ modulo $[p_1,p_2]t$ gives a dual frequency $k$ with
\begin{align}\label{eq:quadratic-s1-poisson-frequency}
|k|\ll X^\ee{[p_1,p_2]t}/{N}.
\end{align}
Indeed, after fixing $p_i,n_i,t,m_1,z$, the factors $\kappa_\mathcal{Q}(p_1t)\kappa_\mathcal{Q}(p_2t)b_\mathcal{Q}(n_1)b_\mathcal{Q}(n_2)$ are constant in the Poisson variable. The derivative bounds for the two $\Xi$ factors and the dyadic cutoff give \eqref{eq:quadratic-s1-poisson-frequency} by repeated integration by parts. The arithmetic coefficients contribute only $X^\ee$.

The complete correlation is Lemma \ref{lem:blomer-li-correlation} with row variables $d_i=-\overline{D}n_i$, where $\overline{D}$ is one integer satisfying $D\overline{D}\equiv1\pmod{[p_1,p_2]t}$.  Thus the two row parameters use a common integer representative. The common sign $\sigma$ in the column can be removed by changing the Poisson variable, which replaces $k$ by $\sigma k$ and leaves all estimates unchanged. Because $p_i\nmid D$, $(d_1-d_2,p_i)=(\overline{D}(n_1-n_2),p_i)=(n_1-n_2,p_i)$, and all local coprimality hypotheses are unchanged.

The frequency $k=0$ is now separated before any average over $k$. When $p_1=p_2$, Lemma \ref{lem:blomer-li-correlation} gives the exact zero-frequency structure $T\ll p_1^2t^{3+\ee}(n_1-n_2,p_1).$  
The subcase $n_1=n_2$ is the full diagonal and is kept by itself. Before the harmless inequality $H\ge1$ is used, its contribution to the quadratic bracket is ${Q^4}/{H^2}$. For $n_1\ne n_2$, we use
\begin{align*}
\sum_{n_1,n_2\ll \f{X^{\ee}{Q^2}}{X}}(n_1-n_2,p)\ll X^{-2+\ee}{Q^4}+X^{-1+\ee}{Q^2}p,
\end{align*}
with the diagonal omitted from the left-hand side when applying the first term. For $k\ne0$, Lemma \ref{lem:blomer-li-correlation} gives the two correlation estimates; the $k$-sum is bounded using \eqref{eq:quadratic-s1-poisson-frequency}.  The three resulting contributions are
\begin{align*}
{Q^4}/{H^2},\quad {Q^5}Q_2/X,\quad Q_1^5Q_2^{7/2},
\end{align*}
corresponding respectively to the complete zero-frequency diagonal, the remaining zero-frequency terms with $p_1=p_2$, and the nonzero-frequency correlations, including the generic $p_1\ne p_2$ terms. Substitution into \eqref{eq:quadratic-s1-dyadic-block}, followed by $H\ge1$, gives
\begin{align*}	
S_1^\mathcal{Q}(N,H)\ll\frac{X^{1/2+\ee}}{Q_1^{3/2}Q_2}\Big(Q^4+\frac{Q^3}{X}Q_1^2Q_2^3+Q_1^5Q_2^{7/2}\Big)^{1/2}.
\end{align*}
In particular, the zero Poisson frequency has not been replaced by the average appropriate to $k\ne0$.
Combining the large- and small-$H$ ranges, and assuming $Q\ge X^{1/2}$, we obtain
\begin{align}	\label{eq:quadratic-s1-final-bound}
S_1^\mathcal{Q}\ll X^\ee\Big(\frac{X^{1/2}Q}{Q_1^{1/2}}+\frac{Q^2}{Q_1}+\frac{X^{1/2}Q}{Q_2^{1/4}}\Big).
\end{align}
On the Jutila moduli coprime to $2D$, Lemma \ref{lem:binary-gauss-reduction} makes the binary Gauss factor a scalar before Cauchy--Schwarz is applied. The resulting correlation estimates therefore have exactly the powers appearing in the divisor case. In Section \ref{sec-S2}, the weighted operator handles the unrestricted Farey moduli, while the explicit zero-frequency mean cancels the original multiplier $1$ up to the three terms in \eqref{eq:quadratic-major-arc-bound}.

\subsection{Completion of the proof}

Let $\ee>0$ be prescribed. Choose $0<\eta<\min(\ee/20,1/44)$ and put
\begin{align}\label{eq:quadratic-parameter-choice}
C_0=X^{1/2-\eta},\quad C=X^{1/2+\eta},\quad Q_1=X^{2/11},\quad Q_2=X^{4/11},\quad Q=X^{6/11}.
\end{align}
We have $C_0<C<Q$ and $Q>X^{1/2}$. The quadratic-form dual sum on $c\le C_0$ is negligible because ${C_0^2}/{X}={X}/{C^2}=X^{-2\eta}.$
The support and row-injectivity conditions follow from
\begin{align*}
\frac{C_0Q_2}{X}=\frac{C_0Q_1^2}{X}=X^{-3/22-\eta},\quad\frac{C_0Q_1}{X}=X^{-7/22-\eta},\quad\frac{Q}{C_0}=X^{1/22+\eta}.
\end{align*}
After $\eta$ is fixed, choose the auxiliary constants in $\delta=X^{-1+\ee_0}$, the $z$-truncation, and the Poisson truncation sufficiently small, in particular with $\ee_0+\eta_0<\eta$ and $\ee_1<\ee_0/10$. Then
\begin{align*}
C_0Q_2\delta=o(1),\quad C_0Q_1\delta=o(1),\quad C_0^2\delta X^{\eta_0}=o(1),\quad C_0Q_1^2\delta X^{\eta_0}=o(1),
\end{align*}
and $4C_0<Q$ for sufficiently large $X$. The harmless powers used in the transform bounds and divisor estimates are also chosen sufficiently small relative to $\eta$ and the prescribed $\ee$. Thus they preserve the effective-block condition $N/(c/m_1)=o(1)$ and are absorbed, together with the loss $X^\eta$, into $X^\ee$ in the final bound. All dependence on the fixed form $\mathcal Q$, including that arising through $D$ and $R$, is absorbed into the implied constant. Consequently, we have
\begin{align*}
S_1^\mathcal{Q}+S_2^\mathcal{Q}\ll X^{21/22+\ee},
\end{align*}
which proves Theorem \ref{T1}.

\subsection*{Acknowledgments}
The authors are very grateful to the referee for some extremely valuable comments and suggestions. 
Hu is supported by the National Natural Science Foundation of China (Grant No.\,12201363) and the Natural Science Foundation of Shandong Province (Grant Nos.\,ZR2022QA047, ZR2024MA053). L\"u is supported by the Key R\&D Program of Shandong Province (No.\,2026CXPT038) and the National Natural Science Foundation of China (No.\,12631001).

\subsection*{Data Availability}
No data was used for the research described in the article.

\subsection*{Conflict of interest}
The authors declare that there is no conflict of interest.

\bibliographystyle{plain}
\bibliography{reference}

@article{AS,
	title = {Exponential sums and {Newton} polyhedra: cohomology and estimates},
	author = {Adolphson, A. and Sperber, S.},
	fjournal = {Annals of Mathematics. Second Series},
	journal = {Ann. of Math. (2)},
	issn = {0003-486X},
	volume = {130},
	number = {2},
	pages = {367--406},
	year = {1989},
	language = {English},
	doi = {10.2307/1971424},
	zbMATH = {4191864},
	Zbl = {0723.14017}
}

@article{Blomer-2004-IMRN,
	author = {Blomer, V.},
	title = {Shifted convolution sums and subconvexity bounds for automorphic {{\(L\)}}-functions},
	fjournal = {IMRN. International Mathematics Research Notices},
	journal = {Int. Math. Res. Not.},
	issn = {1073-7928},
	volume = {2004},
	number = {73},
	pages = {3905--3926},
	year = {2004},
	language = {English},
	doi = {10.1155/S1073792804142505},
	zbMATH = {2140876},
	Zbl = {1066.11019}
}

@article{Blomer-Li-2025,
	author = {V. Blomer and J. X. Li},
	title = {A higher rank shifted convolution problem with applications to {{\(L\)}}-functions},
	journal = {Proc. London Math. Soc.},
	volume = {133},
	number = {3},
	pages = {e70212},
	doi = {https://doi.org/10.1112/plms.70212},
	url = {https://londmathsoc.onlinelibrary.wiley.com/doi/abs/10.1112/plms.70212},
	eprint = {https://londmathsoc.onlinelibrary.wiley.com/doi/pdf/10.1112/plms.70212},
	year = {2026}
}

@article{Blomer--Harcos-2008-Duke-Math,
	author = {Blomer, V. and Harcos, G.},
	title = {The spectral decomposition of shifted convolution sums},
	fjournal = {Duke Mathematical Journal},
	journal = {Duke Math. J.},
	issn = {0012-7094},
	volume = {144},
	number = {2},
	pages = {321--339},
	year = {2008},
	language = {English},
	doi = {10.1215/00127094-2008-038},
	zbMATH = {5317182},
	Zbl = {1246.11108}
}

@article{Duke-Friedlander-Iwaniec--1993-Invent,
	author = {Duke, W. and Friedlander, J. and Iwaniec, H.},
	title = {Bounds for automorphic {{\(L\)}}-functions},
	fjournal = {Inventiones Mathematicae},
	journal = {Invent. Math.},
	issn = {0020-9910},
	volume = {112},
	number = {1},
	pages = {1--8},
	year = {1993},
	language = {English},
	doi = {10.1007/BF01232422},
	url = {https://eudml.org/doc/144092},
	zbMATH = {175636},
	Zbl = {0765.11038}
}

@article{Duke-Friedlander-Iwaniec--1994-Invent,
	author = {Duke, W. and Friedlander, J. B. and Iwaniec, H.},
	title = {A quadratic divisor problem},
	fjournal = {Inventiones Mathematicae},
	journal = {Invent. Math.},
	issn = {0020-9910},
	volume = {115},
	number = {2},
	pages = {209--217},
	year = {1994},
	language = {English},
	doi = {10.1007/BF01231758},
	url = {https://eudml.org/doc/144167},
	zbMATH = {549843},
	Zbl = {0791.11049}
}

@book {goldfeld,
	AUTHOR = {Goldfeld, D.},
	TITLE = {Automorphic forms and {$L$}-functions for the group {${\rm
	GL}(n,\bold R)$}},
	SERIES = {Cambridge Studies in Advanced Mathematics},
	VOLUME = {99},
	PUBLISHER = {Cambridge University Press, Cambridge},
	YEAR = {2006},
	PAGES = {xiv+493},
	ISBN = {978-0-521-83771-2; 0-521-83771-5},
	MRCLASS = {11F55 (11F66 11F70 11R39)},
	MRNUMBER = {2254662},
	MRREVIEWER = {Emmanuel\ P.\ Royer},
	DOI = {10.1017/CBO9780511542923},
	URL = {https://doi.org/10.1017/CBO9780511542923},
}

@article {GL,
	AUTHOR = {Goldfeld, D. and Li, X. Q.},
	TITLE = {Voronoi formulas on {${\rm GL}(n)$}},
	JOURNAL = {Int. Math. Res. Not.},
	FJOURNAL = {International Mathematics Research Notices},
	YEAR = {2006},
	PAGES = {Art. ID 86295, 25},
	ISSN = {1073-7928,1687-0247},
	MRCLASS = {11F66 (11F70 11M41)},
	MRNUMBER = {2233713},
	MRREVIEWER = {Kazuyuki\ Hatada},
	DOI = {10.1155/IMRN/2006/86295},
	URL = {https://doi.org/10.1155/IMRN/2006/86295},
}

@article{Harcos-2003-Math-Ann,
	author = {Harcos, G.},
	title = {An additive problem in the {Fourier} coefficients of cusp forms},
	fjournal = {Mathematische Annalen},
	journal = {Math. Ann.},
	issn = {0025-5831},
	volume = {326},
	number = {2},
	pages = {347--365},
	year = {2003},
	language = {English},
	doi = {10.1007/s00208-003-0421-1},
	zbMATH = {1982574},
	Zbl = {1045.11028}
}

@article{Harcos-Michel,
	author = {Harcos, G. and Michel, P.},
	title = {The subconvexity problem for {Rankin}-{Selberg} {{\(L\)}}-functions and equidistribution of {Heegner} points. {II}},
	fjournal = {Inventiones Mathematicae},
	journal = {Invent. Math.},
	issn = {0020-9910},
	volume = {163},
	number = {3},
	pages = {581--655},
	year = {2006},
	language = {English},
	doi = {10.1007/s00222-005-0468-6},
	url = {infoscience.epfl.ch/record/129122},
	zbMATH = {5013733},
	Zbl = {1111.11027}
}

@article {hb,
	AUTHOR = {Heath-Brown, D. R.},
	TITLE = {Cubic forms in ten variables},
	JOURNAL = {Proc. London Math. Soc. (3)},
	FJOURNAL = {Proceedings of the London Mathematical Society. Third Series},
	VOLUME = {47},
	YEAR = {1983},
	NUMBER = {2},
	PAGES = {225--257},
	ISSN = {0024-6115,1460-244X},
	MRCLASS = {11D72 (11D88 11E76 11G25 11P55)},
	MRNUMBER = {703978},
	MRREVIEWER = {D.\ J.\ Lewis},
	DOI = {10.1112/plms/s3-47.2.225},
	URL = {https://doi.org/10.1112/plms/s3-47.2.225},
}

@article{Holowinsky-2009-Duke-Math,
	author = {Holowinsky, R.},
	title = {A sieve method for shifted convolution sums},
	fjournal = {Duke Mathematical Journal},
	journal = {Duke Math. J.},
	issn = {0012-7094},
	volume = {146},
	number = {3},
	pages = {401--448},
	year = {2009},
	language = {English},
	doi = {10.1215/00127094-2009-002},
	zbMATH = {5505895},
	Zbl = {1218.11089}
}

@article {hooley,
	AUTHOR = {Hooley, C.},
	TITLE = {An asymptotic formula in the theory of numbers},
	JOURNAL = {Proc. London Math. Soc. (3)},
	FJOURNAL = {Proceedings of the London Mathematical Society. Third Series},
	VOLUME = {7},
	YEAR = {1957},
	PAGES = {396--413},
	ISSN = {0024-6115,1460-244X},
	MRCLASS = {10.1X},
	MRNUMBER = {90613},
	MRREVIEWER = {A.\ L.\ Whiteman},
	DOI = {10.1112/plms/s3-7.1.396},
	URL = {https://doi.org/10.1112/plms/s3-7.1.396},
}

@article {HL-exp-sums,
	AUTHOR = {Hu, G. W.  and   L\"{u}, G. S.},
	TITLE = {Exponential sums connected with quadratic forms and their applications.},
	NOTE ={To appear in {\em Acta Math. Sin. (Engl. Ser.)}, 2026},
	}

@article {HL-preprint,
	AUTHOR = {Hu, G. W.  and   L\"{u}, G. S.},
	TITLE = {Shifted convolution sums of {F}ourier coefficients of  \textsc{GL}(3) cusp form.},
	JOURNAL ={Submitted for publication.},
	}

@incollection {js,
	AUTHOR = {Jacquet, H. and Shalika, J.},
	TITLE = {Rankin-{S}elberg convolutions: {A}rchimedean theory},
	SERIES = {Israel Math. Conf. Proc.},
	VOLUME = {2},
	PAGES = {125--207},
	PUBLISHER = {Weizmann, Jerusalem},
	YEAR = {1990},
	MRCLASS = {22E50},
	MRNUMBER = {1159102},
	MRREVIEWER = {Brian\ E.\ Blank},
}

@article {JL,
	AUTHOR = {Jiang, Y. J. and L\"u, G. S.},
	TITLE = {Shifted convolution sums for higher rank groups},
	JOURNAL = {Forum Math.},
	FJOURNAL = {Forum Mathematicum},
	VOLUME = {31},
	YEAR = {2019},
	NUMBER = {2},
	PAGES = {361--383},
	ISSN = {0933-7741,1435-5337},
	MRCLASS = {11F11 (11F27 11F30 11F37)},
	MRNUMBER = {3918446},
	MRREVIEWER = {Yichao\ Zhang},
	DOI = {10.1515/forum-2017-0269},
	URL = {https://doi.org/10.1515/forum-2017-0269},
}

@incollection{Jutila-1992,
	author = {Jutila, M.},
	title = {Transformations of exponential sums},
	booktitle = {Proceedings of the Amalfi conference on analytic number theory, held at Maiori, Amalfi, Italy, from 25 to 29 September, 1989},
	pages = {263--270},
	year = {1992},
	publisher = {Salerno: Universit{\'a} di Salerno},
	language = {English},
	zbMATH = {421684},
	Zbl = {0794.11033}
}

@article{Jutila-1996-Math-Z,
	author = {Jutila, M.},
	title = {The additive divisor problem and its analogs for {Fourier} coefficients of cusp forms. {I}},
	fjournal = {Mathematische Zeitschrift},
	journal = {Math. Z.},
	issn = {0025-5874},
	volume = {223},
	number = {3},
	pages = {435--461},
	year = {1996},
	language = {English},
	doi = {10.1007/PL00004270},
	url = {https://eudml.org/doc/174930},
	zbMATH = {980744},
	Zbl = {0865.11062}
}

@article {KMV,
	AUTHOR = {Kowalski, E. and Michel, P. and VanderKam, J.},
	TITLE = {Rankin-{S}elberg {$L$}-functions in the level aspect},
	JOURNAL = {Duke Math. J.},
	FJOURNAL = {Duke Mathematical Journal},
	VOLUME = {114},
	YEAR = {2002},
	NUMBER = {1},
	PAGES = {123--191},
	ISSN = {0012-7094,1547-7398},
	MRCLASS = {11F66 (11G40)},
	MRNUMBER = {1915038},
	MRREVIEWER = {K.\ Soundararajan},
	DOI = {10.1215/S0012-7094-02-11416-1},
	URL = {https://doi.org/10.1215/S0012-7094-02-11416-1},
}

@article {MS,
	AUTHOR = {Miller, S. D. and Schmid, W.},
	TITLE = {Automorphic distributions, {$L$}-functions, and {V}oronoi
	summation for {${\rm GL}(3)$}},
	JOURNAL = {Ann. of Math. (2)},
	FJOURNAL = {Annals of Mathematics. Second Series},
	VOLUME = {164},
	YEAR = {2006},
	NUMBER = {2},
	PAGES = {423--488},
	ISSN = {0003-486X,1939-8980},
	MRCLASS = {11F66 (11F70 11M41)},
	MRNUMBER = {2247965},
	MRREVIEWER = {Andre\ Reznikov},
	DOI = {10.4007/annals.2006.164.423},
	URL = {https://doi.org/10.4007/annals.2006.164.423},
}

@article{Michel-2004-Ann-Math,
	author = {Michel, P.},
	title = {The subconvexity problem for {Rankin}-{Selberg} {{\(L\)}}-functions and equidistribution of {Heegner} points},
	fjournal = {Annals of Mathematics. Second Series},
	journal = {Ann. of Math. (2)},
	issn = {0003-486X},
	volume = {160},
	number = {1},
	pages = {185--236},
	year = {2004},
	language = {English},
	doi = {10.4007/annals.2004.160.185},
	zbMATH = {2151002},
	Zbl = {1068.11033}
}

@article {Munshi-Duke2013,
	AUTHOR = {Munshi, R.},
	TITLE = {Shifted convolution sums for
	{${\rm GL}(3)\times{\rm GL}(2)$}},
	JOURNAL = {Duke Math. J.},
	FJOURNAL = {Duke Mathematical Journal},
	VOLUME = {162},
	YEAR = {2013},
	NUMBER = {13},
	PAGES = {2345--2362},
	ISSN = {0012-7094,1547-7398},
	MRCLASS = {11F66 (11M41)},
	MRNUMBER = {3127803},
	MRREVIEWER = {M.\ Ram\ Murty},
	DOI = {10.1215/00127094-2371416},
	URL = {https://doi.org/10.1215/00127094-2371416},
}

@incollection{Munshi-2013-divisor,
	author = {Munshi, R.},
	title = {Shifted convolution of divisor function {{\(d_3\)}} and {Ramanujan} {{\(\tau\)}} function},
	booktitle = {The legacy of Srinivasa Ramanujan. Proceedings of the international conference in celebration of the 125th anniversary of Ramanujan's birth, University of Delhi, Delhi, India, December 17--22, 2012},
	isbn = {978-93-80416-13-7},
	pages = {251--260},
	year = {2013},
	publisher = {Mysore: Ramanujan Mathematical Society},
	language = {English},
	zbMATH = {6389649},
	Zbl = {1370.11059}
}

@article{Sarnak-2001-J-Funct-Anal,
	author = {Sarnak, P.},
	title = {Estimates for {Rankin}-{Selberg} {{\(L\)}}-functions and quantum unique ergodicity},
	fjournal = {Journal of Functional Analysis},
	journal = {J. Funct. Anal.},
	issn = {0022-1236},
	volume = {184},
	number = {2},
	pages = {419--453},
	year = {2001},
	language = {English},
	doi = {10.1006/jfan.2001.3783},
	zbMATH = {1656006},
	Zbl = {1006.11022}
}

@article{Pitt-1995-Duke,
	author = {Pitt, N. J. E.},
	title = {On shifted convolutions of {{\(\zeta^3(s)\)}} with automorphic {{\(L\)}}-functions},
	fjournal = {Duke Mathematical Journal},
	journal = {Duke Math. J.},
	issn = {0012-7094},
	volume = {77},
	number = {2},
	pages = {383--406},
	year = {1995},
	language = {English},
	doi = {10.1215/S0012-7094-95-07711-4},
	zbMATH = {784708},
	Zbl = {0855.11024}
}

@misc{Selberg-1965,
	author = {Selberg, A.},
	title = {On the estimation of {Fourier} coefficients of modular forms},
	year = {1965},
	language = {English},
	howpublished = {Proc. {Symp}. {Pure} {Math}. 8:1-15},
	zbMATH = {3229797},
	Zbl = {0142.33903}
}

@article{Tang-2020-BAMS,
	author = {Tang, H. C.},
	title = {A shift convolution sum of {{\(d_3\)}} and the {Fourier} coefficients of {Hecke}-{Maass} forms. {II}},
	fjournal = {Bulletin of the Australian Mathematical Society},
	journal = {Bull. Aust. Math. Soc.},
	issn = {0004-9727},
	volume = {101},
	number = {3},
	pages = {401--414},
	year = {2020},
	language = {English},
	doi = {10.1017/S000497271900100X},
	zbMATH = {7201220},
	Zbl = {1440.11068}
}

@article{Topacogullari-2016,
	author = {Topacogullari, B.},
	title = {The shifted convolution of divisor functions},
	fjournal = {The Quarterly Journal of Mathematics},
	journal = {Q. J. Math.},
	issn = {0033-5606},
	volume = {67},
	number = {2},
	pages = {331--363},
	year = {2016},
	language = {English},
	doi = {10.1093/qmath/haw010},
	zbMATH = {6595938},
	Zbl = {1356.11070}
}

@article {xi-shifted-sum-Forum,
	AUTHOR = {Xi, P.},
	TITLE = {A shifted convolution sum for {$\rm GL(3)\times\rm GL(2)$}},
	JOURNAL = {Forum Math.},
	FJOURNAL = {Forum Mathematicum},
	VOLUME = {30},
	YEAR = {2018},
	NUMBER = {4},
	PAGES = {1013--1027},
	ISSN = {0933-7741,1435-5337},
	MRCLASS = {11F30 (11L07 11M41 11T23)},
	MRNUMBER = {3824803},
	MRREVIEWER = {Bingrong\ Huang},
	DOI = {10.1515/forum-2017-0236},
	URL = {https://doi.org/10.1515/forum-2017-0236},
}

\end{document}